\documentclass[a4paper,12pt]{article}

\unitlength\textwidth
\divide\unitlength by 150\relax

\usepackage{amsmath,amssymb}
\usepackage{bm}

\bmdefine{\sss}{s}
\bmdefine{\vvv}{v}

\newcommand{\NNN}{\mathbb{N}}
\newcommand{\ZZZ}{\mathbb{Z}}
\newcommand{\QQQ}{\mathbb{Q}}
\newcommand{\RRR}{\mathbb{R}}
\newcommand{\KKK}{\mathbb{K}}

\newcommand{\mmmm}{\mathfrak{m}}

\newcommand{\RRRRR}{{\mathcal R}}
\newcommand{\TTTTT}{{\mathcal T}}

\newcommand{\CCCCC}{{\mathcal C}}
\newcommand{\FFFFF}{{\mathcal F}}

\newcommand{\TTTTTn}{{\mathcal T}^{(n)}}
\newcommand{\TTTTTe}{{\mathcal T}^{(\epsilon)}}

\newcommand{\covers}{\mathrel{\cdot\!\!\!>}}
\newcommand{\covered}{\mathrel{<\!\!\!\cdot}}

\newcommand{\define}{\mathrel{:=}}

\newcommand{\gor}{Gorenstein}
\newcommand{\cm}{Cohen-Macaulay}
\newcommand{\sr}{Stanley-Reisner}

\newcommand{\joinirred}{join-irreducible}
\newcommand{\rank}{\mathrm{rank}}
\newcommand{\meet}{\wedge}

\newcommand{\join}{\vee}
\newcommand{\bigjoin}{\bigvee}

\newcommand{\Div}{{\mathrm{Div}}}
\renewcommand{\hom}{{\mathrm{Hom}}}
\newcommand{\interior}{{\mathrm{int}}}
\newcommand{\prehat}[1]{{}^{#1}}

\newcommand{\condn}{{condition N}}
\newcommand{\scn}{{sequence with condition N}}
\newcommand{\sscn}{{sequences with condition N}}

\newcommand{\nudown}{\nu^{\downarrow}}
\newcommand{\nuup}{\nu^{\uparrow}}

\newcommand{\dwfg}{{degree-wise finitely generated}}

\newcommand{\dist}{{\mathrm{dist}}}
\newcommand{\qnmax}{{q^{(n)}_{\max}}}
\newcommand{\qdist}[1]{{q^{(#1)}\dist}}
\newcommand{\qndist}{{q^{(n)}\dist}}
\newcommand{\qmdist}{{q^{(m)}\dist}}
\newcommand{\qedist}{q^{(\epsilon)}\dist}
\newcommand{\qnedist}{q^{(n\epsilon)}\dist}
\newcommand{\qmodist}{q^{(-1)}\dist}

\newcommand{\qn}{q^{(n)}}
\newcommand{\qm}{q^{(m)}}
\newcommand{\qone}{q^{(1)}}
\newcommand{\qmone}{q^{(-1)}}
\newcommand{\qe}{q^{(\epsilon)}}
\newcommand{\qne}{q^{(n\epsilon)}}

\newcommand{\qnred}{$\qn$-reduced}

\newcommand{\qonered}{$\qone$-reduced}
\newcommand{\qmonered}{$\qmone$-reduced}
\newcommand{\qered}{$q^{(\epsilon)}$-reduced}
\newcommand{\qnered}{$q^{(n\epsilon)}$-reduced}

\newcommand{\rkh}{\RRRRR_\KKK[H]}

\newcommand{\cx}{\mathrm{cx}}
\newcommand{\cxf}{\mathrm{cx}_F}
\newcommand{\tcx}{\mathrm{Tcx}}

\newcommand{\pnonmin}{P_\mathrm{nonmin}}
\newcommand{\pnonmax}{P_\mathrm{nonmax}}

\newtheorem{thm}{Theorem}[section]
\newtheorem{fact}[thm]{Fact}
\newtheorem{example}[thm]{Example}
\newtheorem{lemma}[thm]{Lemma}
\newtheorem{cor}[thm]{Corollary}
\newtheorem{definition}[thm]{Definition}

\newtheorem{remark}[thm]{Remark}

\newtheorem{prob}[thm]{Problem}

\newcommand{\bigzerou}{\smash{\lower1.7ex\hbox{\bg 0}}}

\newcommand{\bigastu}{\smash{\lower1.7ex\hbox{\bg *}}}

\newcommand{\refeq}[1]{(\ref{#1})}
\numberwithin{equation}{section}

\newcommand{\mylabel}[1]{{\label{#1}\tt [#1]}}
\let\mylabel=\label

\title{%
Fiber cones, analytic spreads of the canonical and anticanonical
ideals 
and limit Frobenius complexity
of Hibi rings%
}

\author{Mitsuhiro MIYAZAKI\footnote{%
The author is supported partially by 
JSPS KAKENHI Grant Number JP15K04818.}%
}
\date{\normalsize
Department of Mathematics, Kyoto University of Education,\\
1 Fujinomori, Fukakusa, Fushimi-ku, Kyoto, 612-8522, Japan}

\begin{document}


\maketitle

\sloppy

\begin{abstract}
Let $\rkh$ be the Hibi ring over a field $\KKK$ on a finite distributive lattice
$H$, $P$ the set of \joinirred\ elements of $H$ and $\omega$ the canonical ideal of $\rkh$.
We show the powers $\omega^{(n)}$ of $\omega$ in the group of divisors 
$\Div(\rkh)$ is identical with the ordinary powers of $\omega$, 
describe the $\KKK$-vector space basis of $\omega^{(n)}$ for $n\in\ZZZ$.
Further, we show that the fiber cones
$\bigoplus_{n\geq 0}\omega^n/\mmmm\omega^n$ and 
$\bigoplus_{n\geq0}(\omega^{(-1)})^n/\mmmm(\omega^{(-1)})^n$
of $\omega$ and $\omega^{(-1)}$ are sum of the Ehrhart rings,
 defined by sequences of elements of $P$ with a certain condition,
which are polytopal complex version of \sr\ rings.
Moreover, we show that the analytic spread of $\omega$ and $\omega^{(-1)}$ are
maximum of the dimensions of these Ehrhart rings.
Using these facts, we show that the question of Page about Frobenius complexity is
affirmative:
$\lim_{p\to\infty}\cxf(\rkh)=\dim(\bigoplus_{n\geq0}\omega^{(-n)}/\mmmm\omega^{(-n)})-1$,
where $p$ is the characteristic of the field $\KKK$.
\end{abstract}

\section{Introduction}

Lyubeznik and Smith \cite{ls} defined the ring of Frobenius operators:
let $R$ be a commutative ring with prime characteristic $p$ and $M$ an 
$R$-module.
Let $\prehat e M$ denote the $R$-module whose additive module structure
is that of $M$ and the action of $R$ is defined by $e$-times iterated 
Frobenius map: $r\cdot m=r^{p^e}m$, where the right hand side is 
the original action of $R$ on $M$.
$\hom_R(M,\prehat e M)$ is an additive group, which is denoted
$\FFFFF^e(M)$.
Since for any $\varphi\in\hom_R(M,\prehat e M)$ and
$\phi\in\hom(M,\prehat {e'} M)$, $\phi\circ\varphi\in\hom_R(M,\prehat {e+e'}M)$,
we see that $\bigoplus_{e\geq0}\FFFFF^e(M)$ has a structure of noncommutative
ring which is denoted $\FFFFF(M)$ and called the ring of Frobenius operators
on $M$ in \cite{ls}.

In \cite{ls}, they studied the relation of finite generation of $\FFFFF(M)$
over $\FFFFF^0(M)$ and well behavior of tight closure operation,
e.g.\ commutativity of tight closure of an ideal and 
localization of a ring.
Despite the fact that it is now known that tight closure does not
commute with localization \cite{bm},
problem of finite generation of $\FFFFF(M)$, especially the case where
$R$ is a local ring and $M$ is the injective hull $E$ of the residue field of $R$
is important: see \cite{kssz}.
Moreover, Enescu and Yao \cite{ey1} defined the Frobenius complexity of a local
ring by taking $\log_p$ of the complexity of $\FFFFF(E)$:
the complexity of an $\NNN$-graded ring is a measure of infinite generation 
over its degree 0 part.
They took $\log_p$ in the definition of Frobenius complexity because there is
substantial evidence that, in important cases, there is a limit as $p\to \infty$
of Frobenius complexity.
Above all they showed in \cite{ey2} that if $m>n\geq 2$, then the determinantal
ring obtained by modding out the 2-minors of an $m\times n$ matrix of indeterminates with
base field prime characteristic $p$ has limit Frobenius complexity $m-1$ as $p\to\infty$.

Page \cite{pag} generalized this result to non-\gor\ anticanonical level Hibi rings:
let $\rkh$ be  a Hibi ring over a field $\KKK$ of characteristic $p$ on 
a distributive lattice $H$ and $P$ the set of \joinirred\ elements of $H$.
Then the Frobenius complexity of $\rkh$ approaches to $\#\pnonmin$ as $p\to\infty$,
where $\pnonmin=\{z\in P\mid z$ is not in any maximal chain of $P$ of minimal length$\}$.
See Definition \ref{def:anticanonical level} for the definition of anticanonical level 
property.

In the case where $\rkh$ is anticanonical level, $\#\pnonmin$ is equal to the 
analytic spread of $\omega^{(-1)}$ minus 1, where $\omega$ is the canonical module
of $\rkh$ and $\omega^{(-1)}$ is the inverse element of $\omega$ in $\Div(\rkh)$.
Thus, Page raised a question if the limit of Frobenius complexity of an arbitrary
non-\gor\ Hibi ring is equal to the analytic spread of $\omega^{(-1)}$ minus 1
as $p\to\infty$ \cite[Question 5.1]{pag}.

The main purpose of this paper is to answer this question affirmatively
(see Theorem \ref{thm:f comp}).
In order to accomplish this task, we first analyze the fiber cone of $\omega^{(-1)}$.
Since the treatment is the same for $\omega$, we study the fiber cones of $\omega$
and $\omega^{(-1)}$ simultaneously.
We show that the fiber cone of $\omega$ (resp.\ $\omega^{(-1)}$) is a finite sum
of Ehrhart rings each of which is defined by a certain 
``\scn" (see Definition \ref{def:condn})
and express the analytic spread of $\omega$ (resp.\ $\omega^{(-1)}$) by the
dimensions of the Ehrhart rings defined by these sequences.
This expression, which is described by a polytopal complex, is interesting in its own right.
After this, we show that the Frobenius complexity of $\rkh$ approaches to the analytic
spread of $\omega^{(-1)}$ minus 1 by using the expression above.

This paper is organized as follows.
First 
in \S \ref{sec:pre},
we recall the definition and basic facts of Hibi rings,
study the $n$-th power $\omega^{(n)}$
of the canonical ideal of $\rkh$ in $\Div(\rkh)$, where $n\in\ZZZ$ and $\rkh$ is the Hibi
ring over $\KKK$ on a finite distributive lattice $H$.
We describe Laurent monomials in $\omega^{(n)}$ for $n\in\ZZZ$ and show that for $n>0$,
$\omega^{(n)}=\omega^n$ and $\omega^{(-n)}=(\omega^{(-1)})^n$.
See Theorem \ref{thm:omega n basis}.
Though these results are obtained by Page \cite[Corollary 3.1 and Proposition 3.2]{pag}
for the case of negative powers, our proof is more down to earth 
and treat the cases of positive and negative powers simultaneously.

Next in \S \ref{sec:qnred}, we recall the notion of a \scn\ \cite[Definition 3.1]{mo},
define the notion of \qnred\ \scn,
where $n\in\ZZZ$.
See Definitions \ref{def:condn} and \ref{def:qnred}.
We show that the Laurent monomial $\prod_{x}T_x^{\nu(x)}$, where $\nu$ is a map from
the set $P$ of \joinirred\ elements of $H$ to $\ZZZ$, is a generator of 
$\omega^{(n)}$ if there is a \qnred\ \scn\ with a certain 
condition related to $\nu$.
Conversely, we construct for each \qnred\ \scn, maps $\nudown$ and $\nuup$ form $P$ to $\ZZZ$ 
such that the Laurent monomials $\prod_{x\in P}T_x^{\nudown(x)}$ and $\prod_{x\in P}T_x^{\nuup(x)}$
are generators of $\omega^{(n)}$.
From this, we deduce that $\rkh$ is level (resp.\ anticanonical level) if and only if 
\qonered\ (resp. \qmonered) \scn\ is the empty sequence only.
Further, we show the degrees of the generators of $\omega^{(n)}$ are consecutive integers,
i.e., if there are generators of degrees $d_1$ and $d_2$ of $\omega^{(n)}$ with $d_1<d_2$,
then for any integer $d$ with $d_1\leq d\leq d_2$, 
there is a generator of $\omega^{(n)}$ with degree $d$.

After these preparations, 
we define in
\S \ref{sec:conv polytope},
for each \qered\ 
\scn\ an integral convex polytope whose
Ehrhart ring is standard, i.e., generated by elements of degree 1 over the base field,
where $\epsilon=\pm1$.
We express the dimension of this convex polytope by the word of poset and the
\qered\ \scn\ which defines this convex polytope.
As a special case, we show that if the \qonered\ (resp.\ \qmonered) 
sequence under consideration is an 
empty sequence, then the dimension of this convex polytope is
$\#\pnonmax$ (resp.\ $\#\pnonmin$), where
$\pnonmax=\{z\in P\mid z$ is not in any 
chain of $P$ of maximal length$\}$.

In \S \ref{sec:anal spread},
we show that the Ehrhart ring defined by the convex polytope above is 
isomorphic to a graded subalgebra of the fiber cone
$\bigoplus_{n\geq 0}\omega^n/\mmmm\omega^n$ (resp.\ 
$\bigoplus_{n\geq0}(\omega^{(-1})^n/\mmmm(\omega^{(-1)})^n)$) of $\omega$ (resp.\ $\omega^{(-1)}$)
if $\epsilon=1$ (resp.\ $\epsilon=-1$),
where $\mmmm$ is the irrelevant maximal ideal of $\rkh$.
Further, we show that 
$\bigoplus_{n\geq 0}\omega^n/\mmmm\omega^n$ (resp.\ 
$\bigoplus_{n\geq0}(\omega^{(-1})^n/\mmmm(\omega^{(-1)})^n)$) 
is the sum of finite number of these types of subalgebras.
Since the dimension of a graded ring is computed by the Hilbert function,
we conclude that the analytic spread of $\omega$ (resp. $\omega^{(-1)}$) is the maximum of the
dimensions of these Ehrhart rings.
We also note that glueing of these Ehrhart rings in
$\bigoplus_{n\geq 0}\omega^n/\mmmm\omega^n$ (resp.\ 
$\bigoplus_{n\geq0}(\omega^{(-1})^n/\mmmm(\omega^{(-1)})^n)$) 
is a generalization of \sr\ rings to polytopal complexes.

In \S \ref{sec:comp}, we recall the definition of complexity of
(not necessarily commutative) $\NNN$-graded ring and Frobenius complexity.
We also define the notion of strong left $R$-skew algebra and show that 
if $A=\bigoplus_{n\geq0}A_n$ is a strong left $A_0$-skew algebra and $A_0$ is a commutative
local ring with maximal ideal $\mmmm$, then $\mmmm A$ is a graded two sided ideal of $A$
and the complexity of $A$ and $A/\mmmm A$ coincide.

In \S \ref{sec:t const}, we recall the operation T-construction defined by
Katzman et al.\ \cite{kssz} and define the T-complexity of a commutative 
$\NNN$-graded ring of characteristic $p$.
By the result of Katzman et al.\ \cite[Theorem 3.3]{kssz} and the results of previous
sections, we see that the Frobenius complexity of a Hibi ring can be computed by
the T-complexities of Ehrhart rings appeared in \S \ref{sec:anal spread}.
We state key lemmas to compute the limit T-complexity of Ehrhart rings.

Finally in \S \ref{sec:frob}, by using the results up to the previous section,
we show that the Frobenius complexities of Hibi rings approaches to analytic spread
of the anticanonical ideal minus 1.


\section{Posets and Hibi rings}
\mylabel{sec:pre}

In this paper, all rings and algebras are assumed to have identity element
and, up to \S \ref{sec:anal spread}, assumed to 
be commutative 
unless stated otherwise.
We also assume that a ring homomorphism maps the identity element to the identity element.
We denote by $\NNN$ the set of nonnegative integers, by
$\ZZZ$ the set of integers, by
$\RRR$ the set of real numbers 
by $\RRR_{>0}$ the set of positive real numbers and
by $\RRR_{\geq0}$ the set of nonnegative real numbers.
We use letter $p$ to express a prime number.

We denote the cardinality of a set $X$ by $\#X$.
For two sets $X$ and $Y$, we denote by $X\setminus Y$ the set 
$\{x\in X\mid x\not\in Y\}$.
We use this notation not only the case where $X\supset Y$ but also the case
where $X\not\supset Y$.
We denote the set of maps from $X$ to $Y$ by $Y^X$.
If $X$ is a finite set, we identify $\RRR^X$ the Euclidean space $\RRR^{\#X}$.

Next we recall some definitions concerning finite partially
ordered sets (poset for short).
Let $Q$ be a finite poset.
A chain in $Q$ is a totally ordered subset of $Q$.
For a chain $X$ in $Q$, we define the length of $X$ as $\#X-1$.
The maximum length of chains in $Q$ is called the rank of $Q$ and denoted  $\rank Q$.
If every maximal chain of $Q$ has the same length, we say that $Q$ is pure.
If $I\subset Q$ and
$x\in I$, $y\in Q$, $y\leq x\Rightarrow y\in I$,
then we say that $I$ is a poset ideal of $Q$.
If $x$, $y\in Q$, $x<y$ and there is no $z\in Q$ with $x<z<y$,
we say that $y$ covers $x$ and denote 
$x\covered y$ or $y\covers x$.
For $x$, $y\in Q$ with $x\leq y$, we set
$[x,y]_Q\define\{z\in Q\mid x\leq z\leq y\}$.
We denote $[x,y]_Q$ 
as $[x,y]$ 
if there is no fear of confusion.
Let $\infty$ be a new element which is not contained in $Q$.
We define a new poset $Q^+$ whose base set is $Q\cup\{\infty\}$ and
$x<y$ in $Q^+$ if and only if $x$, $y\in Q$ and $x<y$ in $Q$
or $x\in Q$ and $y=\infty$.

\begin{definition}\rm
Let $Q$ be an arbitrary poset and let
$x$ and $y$ be elements of $Q$ with $x\leq y$.
A saturated chain from $x$ to $y$ is a sequence of elements 
$z_0$, $z_1$, \ldots, $z_t$  of $Q$ such that
$$
x=z_0\covered z_1\covered \cdots\covered z_t=y.
$$
\end{definition}
Note that the length of the chain $z_0$, $z_1$, \ldots, $z_t$ is $t$.

\begin{definition}\rm
\mylabel{def:dist}
Let $Q$, $x$ and $y$ be as above.
We define 
$\dist(x,y)\define\min\{t\mid$ there is a saturated chain from $x$ to $y$
with length $t.\}$
and call $\dist(x,y)$ the distance of $x$ and $y$.
Further, for $n\in \ZZZ$, we define 
$\qndist(x,y)\define\max\{nt\mid$
there is a saturated chain from $x$ to $y$
with length $t.\}$
and call $\qndist(x,y)$ the $n$-th quasi-distance of $x$ and $y$.
\end{definition}
Note that $\qdist{-1}(x,y)=-\dist(x,y)$ and
$\qdist{1}(x,y)=\rank([x,y])$.
Note also that $\qndist(x,z)+\qndist(z,y)\leq\qndist(x,y)$
for any $x$, $z$, $y$ with $x\leq z\leq y$.
Further, $\qndist(x,y)=n$ if $x\covered y$,
$\qndist(x,x)=0$ and $\qdist{mn}(x,y)=m\qndist(x,y)$ for any positive integer $m$.

\begin{definition}\rm
\mylabel{def:tn}
For a poset $Q$ and $n\in \ZZZ$, we set $\TTTTTn(Q)\define\{
\nu\colon Q^+\to\ZZZ\mid \nu(\infty)=0$,
$\nu(x)-\nu(y)\geq n$ if $x\covered y$ in $Q^+\}$.
\end{definition}
Note that if $x$ is a maximal element of $Q$ and $\nu\in\TTTTTn(Q)$, then $\nu(x)\geq n$
since $x\covered \infty$ in $Q^+$.
Note also that if $\nu\in\TTTTTn(Q)$, $x$, $y\in Q^+$ and $x\leq y$, then
$\nu(x)-\nu(y)\geq\qndist(x,y)$.

In the following, we identify a map $\nu\colon Q^+\to \RRR$ with $\nu(\infty)=0$ with
the restriction $\nu\mid_Q\colon Q\to \RRR$.

Next we define operations of maps from a set to $\ZZZ$.

\begin{definition}\rm
Let $X$ be a set.
For $\nu$, $\nu'\in \ZZZ^X$ and a positive integer $n$, we define
maps $\nu\pm\nu'$, $\max\{\nu,\nu'\}$, $\min\{\nu,\nu'\}$, $n\nu$ and 
$\lfloor\frac{\nu}{n}\rfloor\in\ZZZ^X$
by
$(\nu\pm\nu')(x)=\nu(x)\pm\nu'(x)$,
$\max\{\nu,\nu'\}(x)=\max\{\nu(x),\nu'(x)\}$,
$\min\{\nu,\nu'\}(x)=\min\{\nu(x),\nu'(x)\}$,
$(n\nu)(x)=n\nu(x)$ and
$\lfloor\frac{\nu}{n}\rfloor(x)=\lfloor\frac{\nu(x)}{n}\rfloor$
for $x\in X$, where $\lfloor\frac{\nu(x)}{n}\rfloor$ is the largest integer
not exceeding $\frac{\nu(x)}{n}$.
\end{definition}
We note the following fact which is easily proved.

\begin{lemma}
\mylabel{lem:op nu}
Let $m$, $\ell$ be integers 
and $n$ an integer greater than 1.
Suppose that $\nu_1$, $\nu'_1\in\TTTTT^{(m)}(Q)$,
$\nu_2\in\TTTTT^{(\ell)}(Q)$, $\nu\in\TTTTTn(Q)$
and $\nu'\in\TTTTT^{(-n)}(Q)$.
Then $\nu_1+\nu_2\in\TTTTT^{(m+\ell)}(Q)$,
$\max\{\nu_1,\nu'_1\}$,
$\min\{\nu_1,\nu'_1\}\in\TTTTT^{(m)}(Q)$,
$n\nu_1\in\TTTTT^{(nm)}(Q)$,
$\lfloor\frac{\nu}{n}\rfloor\in\TTTTT^{(1)}(Q)$,
$\nu-\lfloor\frac{\nu}{n}\rfloor\in\TTTTT^{(n-1)}(Q)$.
$\lfloor\frac{\nu'}{n}\rfloor\in\TTTTT^{(-1)}(Q)$ and
$\nu'-\lfloor\frac{\nu'}{n}\rfloor\in\TTTTT^{(-n+1)}(Q)$.
\end{lemma}

Here we note the following fact.
Let $R$ be a Noetherian normal domain and $I$ a fractional ideal.
$I$ is said to be divisorial if $R:_{Q(R)}(R:_{Q(R)}I)=I$, i.e.,
$I$ is reflexive as an $R$-module,
where $Q(R)$ is the fraction field of $R$.
It is known that the set of divisorial ideals form a group,
denoted $\Div(R)$, by the operation
$I\cdot J\define R:_{Q(R)}(R:_{Q(R)}IJ)$.
We denote the $n$-th power of $I$ in this group $I^{(n)}$, 
where $n\in\ZZZ$.
Note that if $I\subsetneq R$, then $I^{(n)}$ is identical with the 
$n$-th symbolic power of $I$.
Note also that the inverse element of $I$ in $\Div(R)$ is $R:_{Q(R)}I$.

Suppose further that $R$ is a 
standard graded ring over a field $\KKK$
(resp.\ 
affine semigroup ring
generated by Laurent monomials 
in the Laurent polynomial ring $\KKK[X_1^{\pm1}, \ldots, X_s^{\pm1}]$
over $\KKK$
with weight so that $R$ is a standard graded ring,
where $\KKK$ is a field and $X_1$, \ldots, $X_s$ are indeterminates).
Let $I$ be a divisorial ideal generated by 
homogeneous elements
(resp.\ 
Laurent monomials)
$m_1$, \ldots, $m_\ell$.
Then $R:_{Q(R)}I=\bigcap_{i=1}^\ell Rm_i^{-1}$.
Thus, $R:_{Q(R)}I$ is an $R$-submodule of 
$S^{-1}R$ generated by homogeneous elements, 
where $S=\{x\in R\mid x\neq0$, $x$ is a homogeneous element$\}$
(resp.\
of $\KKK[X_1^{\pm1},\ldots, X_s^{\pm1}]$
generated by Laurent monomials).
Therefore, the set of divisorial ideals generated by 
homogeneous elements in $S^{-1}R$
(resp.\
Laurent monomials)
 form a subgroup of $\Div(R)$.
It is known that the canonical module is reflexive and isomorphic to an ideal.
Therefore
$\omega\in\Div(R)$.
Thus, if the  canonical module $\omega$ of
$R$ is isomorphic to an ideal of $R$ generated by 
homogeneous elements
(resp.\
Laurent monomials),
the inverse element $\omega^{(-1)}$ of $\omega$ in $\Div(R)$ 
is also an $R$-submodule
of $S^{-1}R$
(resp.\
$\KKK[X_1^{\pm1},\ldots, X_s^{\pm1}]$)
generated by 
homogeneous elements in $S^{-1}R$
(resp.\
Laurent monomials).

Taking into account of this fact,
we recall the definition of level (resp.\ anticanonical level) property.

\begin{definition}[\cite{sta1,pag}]\rm
\mylabel{def:anticanonical level}
Let $R$ be a standard graded \cm\ algebra over a field.
If the degree of all the generators of the canonical module $\omega$ of $R$ are the same,
then we say that $R$ is level.
Moreover, if $R$ is normal (thus, is a domain) 
and the degree of all the generators of $\omega^{(-1)}$ 
are the same, we say that $R$ is anticanonical level.
\end{definition}
As is noted in \cite[Example 3.4]{pag}, level property does not imply anticanonical 
level property nor anticanonical level property does not imply level property.

Now we recall the definition of a Hibi ring.
A lattice is a poset $L$ such that for any elements $\alpha$ and $\beta\in L$,
there are the minimum upper bound of $\{\alpha,\beta\}$, denoted $\alpha\join\beta$
and the maximum lower bound of $\{\alpha,\beta\}$, denoted $\alpha\meet\beta$.
A lattice $L$ is distributive if 
$\alpha\meet(\beta\join\gamma)=(\alpha\meet\beta)\join(\alpha\meet\gamma)$ and
$\alpha\join(\beta\meet\gamma)=(\alpha\join\beta)\meet(\alpha\join\gamma)$
for any $\alpha$, $\beta$ and $\gamma\in L$.

Let $\KKK$ be a field,
$H$ a finite distributive lattice with unique minimal element
$x_0$, $P$ the set of \joinirred\ elements of $H$, i.e.,
$P=\{\alpha \in H\mid \alpha=\beta\join\gamma\Rightarrow
\alpha=\beta$ or $\alpha=\gamma\}$.
Note that we treat $x_0$ as a \joinirred\ element.
It is known that $H$ is isomorphic to the set of nonempty poset ideals of $P$
ordered by inclusion.

Let $\{T_x\}_{x\in P}$ be a family of indeterminates indexed by $P$.

\begin{definition}[{\cite{hib}}]
\rm
$\rkh\define \KKK[\prod_{x\in I}T_x\mid I$ is a nonempty poset ideal of $P]$.
\end{definition}
It is easily verified that if we set $\alpha=\bigjoin_{x\in I}x$
for a nonempty poset ideal $I$, then 
$I=\{x\in P\mid x\leq \alpha$ in $H\}$.
Further, for $\alpha \in H$, $\{x\in P\mid x\leq \alpha\}$ is a nonempty poset ideal
of $P$.
Thus, 
$\rkh=\KKK[\prod_{x\leq \alpha}T_x\mid \alpha\in H]$.

$\rkh$ is called the Hibi ring over $\KKK$ on $H$ nowadays.
Hibi \cite[\S 2 b)]{hib} showed that $\rkh$ is a normal affine 
semigroup ring and thus is \cm\ by the result of Hochster \cite{hoc}.
Further, he showed \cite[\S 3 d)]{hib} that $\rkh$ is \gor\ if and only if 
$P$ is pure.
Moreover, by setting $\deg T_{x_0}=1$ and $\deg T_x=0$ for $x\in P\setminus\{x_0\}$,
$\rkh$ is a standard graded $\KKK$-algebra.
We denote the graded canonical module of $\rkh$ by $\omega$.

For $\nu\colon P\to\ZZZ$, we denote the Laurent monomial
$\prod_{x\in P}T_x^{\nu(x)}$ by $T^\nu$.
Note that $\deg T^\nu=\nu(x_0)$.
It is shown by Hibi \cite{hib} and is easily verified that
$$
\rkh=\bigoplus_{\nu\in\TTTTT^{(0)}(P)}\KKK T^\nu
$$
and  therefore by the description of the canonical module of a normal
affine semigroup ring by Stanley \cite[p.\ 82]{sta2}, we see that
$$
\omega=\bigoplus_{\nu\in\TTTTT^{(1)}(P)}\KKK T^\nu.
$$
We call this ideal the canonical ideal of $\rkh$ and
$\omega^{(-1)}$ the anticanonical ideal of $\rkh$.

Next we state the following

\begin{lemma}
\mylabel{lem:exist n}
Let $x$ and $y$ be elements of $P^+$ with $x\covered y$ and $n\in\ZZZ$.
Then there exists $\nu\in\TTTTTn(P)$ such that
$\nu(x)-\nu(y)=n$.
\end{lemma}
\begin{proof}
For $z\in P^+$,
set
$$
\nu(z)=
\left\{\begin{array}{ll}
\qndist(z,\infty)\quad&\mbox{if $z\not\leq y$,}
\\
\max\{\qndist(z,\infty),\qndist(x,\infty)-n+\qndist(z,y)\}\quad&\mbox{if $z\leq y$.}
\end{array}
\right.
$$
Then it is easily verified that $\nu$ satisfies the required condition.
\end{proof}

Now we state the following.

\begin{thm}
\mylabel{thm:omega n basis}
For a positive integer $n$,
\begin{eqnarray*}
\omega^n&=&\omega^{(n)}=\bigoplus_{\nu\in\TTTTTn(P)}\KKK T^\nu,
\\
(\omega^{(-1)})^n&=&\omega^{(-n)}=\bigoplus_{\nu\in\TTTTT^{(-n)}(P)}\KKK T^\nu.
\end{eqnarray*}
\end{thm}
\begin{proof}
Let $\nu$ be an arbitrary element of $\TTTTT^{(-n)}(P)$
and let $\nu_1$, \ldots, $\nu_n$ be arbitrary elements in $\TTTTT^{(1)}(P)$.
Then $\nu+\nu_1+\cdots+\nu_n\in\TTTTT^{(0)}(P)$.
Therefore,
$$
\bigoplus_{\nu\in\TTTTT^{(-n)}(P)}\KKK T^\nu\subset
(\rkh\colon \omega^n)=\omega^{(-n)}.
$$
In order to prove the converse inclusion, first note that $\omega^{(-n)}$
is a $\ZZZ^{\#P}$-graded 
$\rkh$-submodule of the Laurent polynomial ring 
$\KKK[T_x^{\pm1}\mid x\in P]$ and therefore a $\KKK$-vector subspace of
$\KKK[T_x^{\pm1}\mid x\in P]$ which has a basis consisting of Laurent monomials.

Let $T^\nu$, $\nu\colon P\to \ZZZ$, be an arbitrary Laurent monomial in $\omega^{(-n)}$.
We extend $\nu$ to a map from $P^+$ to $\ZZZ$ by setting $\nu(\infty)=0$.
Let $x$ and $y$ be arbitrary elements of $P^+$ with $x\covered y$.
Then by Lemma \ref{lem:exist n}, we see that there is $\nu'\in\TTTTT^{(1)}(P)$ such that
$\nu'(x)-\nu'(y)=1$.
Since $(T^{\nu'})^n\in\omega^n$, we see that
$$
T^{\nu+n\nu'}=T^\nu (T^{\nu'})^n\in \rkh.
$$
Thus
$
(\nu+n\nu')(x)-(\nu+n\nu')(y)\geq 0
$
and we see that
$\nu(x)-\nu(y)\geq -n$.

Since $x$ and $y$ are arbitrary, we see that $\nu\in\TTTTT^{(-n)}(P)$.
Thus we see that
$$
\omega^{(-n)}\subset\bigoplus_{\nu\in\TTTTT^{(-n)}(P)}\KKK T^\nu
$$
and therefore
$$
\omega^{(-n)}=\bigoplus_{\nu\in\TTTTT^{(-n)}(P)}\KKK T^\nu.
$$

From this fact, we can show that
$$
\omega^{(n)}=(\rkh\colon \omega^{(-n)})=\bigoplus_{\nu\in\TTTTT^{(n)}(P)}\KKK T^\nu
$$
by a similar way.

Next assume that $\nu$ is an arbitrary element of $\TTTTTn(P)$.
By using Lemma \ref{lem:op nu} repeatedly, we see that there are 
$\nu_1$, \ldots, $\nu_n\in\TTTTT^{(1)}(P)$ such that
$\nu=\nu_1+\cdots+\nu_n$.
Therefore, $T^\nu\in\omega^n$.
Since $\nu$ is an arbitrary element of $\TTTTTn(P)$, we see that
$$
\omega^{(n)}=\bigoplus_{\nu\in\TTTTT^{(n)}(P)}\KKK T^\nu\subset\omega^n.
$$
Thus, we see that $\omega^{(n)}=\omega^n$, since the converse inclusion 
holds in general.
We see that $(\omega^{(-1)})^n=\omega^{(-n)}$ by the same way.
\end{proof}

\begin{remark}\rm
By Theorem \ref{thm:omega n basis}, we see that symbolic Rees algebras
$R=\bigoplus_{n\geq 0}\omega^{(n)}$
and
$R'=\bigoplus_{n\geq 0}\omega^{(-n)}$
are ordinary Rees algebras and therefore Noetherian.
Thus, by applying the result of Goto et al. 
\cite[Theorems (4.5) and (4.8)]{ghnv}
to $\omega$ and $\omega^{(-1)}$,
we see that $R'$ is \gor\
and the canonical module of $R$ is isomorphic to $\omega^{(2)}R$.

In our case, we can describe the canonical modules of these rings explicitly.
Let $X$ be a new indeterminate and we embed the above rings in the Laurent
polynomial ring $\KKK[T_x^{\pm1}\mid x\in P][X^{\pm1}]$ by identifying $R$ with
$\bigoplus_{n\geq 0}\omega^{(n)}X^n=\bigoplus_{n\geq0\atop \nu\in\TTTTTn(P)}\KKK T^\nu X^n$
 and $R'$ with $\bigoplus_{n\geq 0}\omega^{(-n)}X^{-n}=
\bigoplus_{n\geq 0\atop \nu\in\TTTTT^{(-n)}(P)}\KKK T^\nu X^{-n}$.
Then we see that these rings are normal by the Hochster's criterion \cite{hoc} and therefore \cm.
Further, by the Stanley's description of the canonical module of a normal affine semigroup ring,
we see that the canonical module of $R$ (resp.\ $R'$) is 
$\bigoplus_{n>0\atop \nu\in\TTTTT^{(n+1)}(P)}\KKK T^\nu X^n$
(resp.\ $\bigoplus_{n>0\atop \nu\in\TTTTT^{(-n+1)}(P)}\KKK T^\nu X^{-n}$).
Thus, we see by Theorem \ref{thm:omega n basis} that the canonical module of $R'$ is generated 
by $X^{-1}$.
Further, we see that the  canonical module of $R$ is generated by
$\{T^\nu X\mid \nu\in\TTTTT^{(2)}(P)\}$.
\end{remark}

\section{Generators of $\omega^{(n)}$ and \qnred\ sequences with \condn}

\mylabel{sec:qnred}

In this section, we state a characterization of a Laurent monomial to be a generator
of $\omega^{(n)}$, where $n\in\ZZZ$.

First, we introduce an order on $\TTTTTn(P)$ and describe generators of $\omega^{(n)}$
with it, where $n\in\ZZZ$.
Since $\omega^{(n)}$ is a finitely generated $\ZZZ^{\#P}$-graded $\rkh$-submodule of the
Laurent polynomial ring $\KKK[T_x^{\pm1}\mid x\in P]$, there is a unique minimal set of
Laurent monomials which generate $\omega^{(n)}$ as an $\rkh$-module.
We call an element of this set a generator of $\omega^{(n)}$.
By Theorem \ref{thm:omega n basis},
we see that for $\nu\in\TTTTTn(P)$, $T^\nu$ is a generator of $\omega^{(n)}$
if and only if there are no $\nu_1\in\TTTTTn(P)$ and $\nu_2\in\TTTTT^{(0)}(P)$ such that
$\nu\neq\nu_1$ and $\nu=\nu_1+\nu_2$.
On account of this fact, we make the following.

\begin{definition}\rm
\mylabel{def:order ttttt}
Let $n\in\ZZZ$ and $\nu$, $\nu'\in\TTTTTn(P)$.
We define the relation $\leq$ on $\TTTTTn(P)$ by
$$
\nu\leq\nu'\iff\nu'-\nu\in\TTTTT^{(0)}(P).
$$
\end{definition}
It is easily verified that $\leq$ is an order relation on $\TTTTT^{(n)}(P)$.
Further, by the above argument, for $\nu\in\TTTTT^{(n)}(P)$, $T^\nu$
is a generator of $\omega^{(n)}$ if and only if $\nu$ is a minimal element of 
$\TTTTT^{(n)}(P)$.

In the following of this section, we fix $n\in\ZZZ$.
First we state the following (cf. \cite[Definition 3.1]{mo}).

\begin{definition}\rm
\mylabel{def:condn}
We say a (possibly empty) sequence $y_0$, $x_1$, $y_1$, $x_2$, \ldots, $y_{t-1}$, $x_t$
of elements $P\setminus\{x_0\}$ satisfies \condn\ if
\begin{enumerate}
\item
$y_0>x_1<y_1>x_2<\cdots<y_{t-1}>x_t$ and
\item
$y_i\not\geq x_j$ if $i\leq j-2$.
\end{enumerate}
\end{definition}
When considering a \scn, we add $x_0$ at the beginning and $y_t=\infty$ at the end of the
sequence and consider a sequence $x_0$, $y_0$, \ldots, $x_t$, $y_t$.
In particular, when $t=0$, we consider a sequence $x_0$, $\infty$.

In order to simplify description, we set the following.

\begin{notation}
Let
$w_0$, $z_0$, $w_1$, $z_1$, \ldots, $w_\ell$, $z_\ell$ be elements of $P^+$ with
$w_0<z_0>w_1<z_1>\cdots>w_\ell<z_\ell$.
We set
$$
\qn(w_0,z_0,w_1,z_1,\ldots,w_\ell,z_\ell)\define
\sum_{i=0}^\ell\qndist(w_i,z_i)-\sum_{i=1}^{\ell}\qndist(w_{i},z_{i-1}).
$$
\end{notation}
%
%
Next we define the following property of a \scn.

\begin{definition}\rm
\mylabel{def:qnred}
Let $m$ be an integer and
$y_0$, $x_1$, \ldots, $y_{t-1}$, $x_t$ a \scn.
Set $y_t=\infty$.
If for any $0\leq i<j\leq t$ with $x_i\leq y_j$,
$\qmdist(x_i,y_j)<\qm(x_i,y_i,\ldots,x_j,y_j)$,
we say that $y_0$, $x_1$, \ldots, $y_{t-1}$, $x_t$ is $\qm$-reduced.
We treat the empty sequence as a $\qm$-reduced \scn.
\end{definition}
Note that a sequence $y_0$, $x_1$, \ldots, $y_{t-1}$, $x_t$ with \condn\ is \qonered\ 
(resp.\ \qmonered) if and only if it is $\qm$-reduced 
(resp.\ $q^{(-m)}$-reduced) for any $m>0$.
\begin{example}
\rm
If there is a part of the sequence of the following form
$$
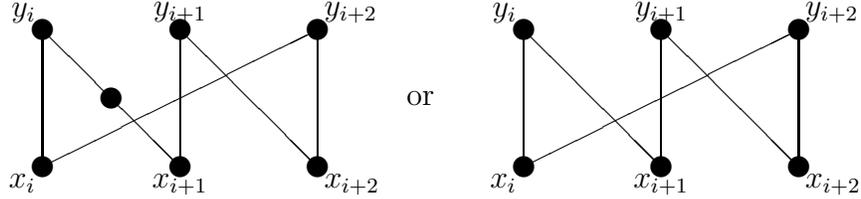
\begin{picture}(60,40)

\put(10,10){\circle*{3}}
\put(10,30){\circle*{3}}
\put(30,10){\circle*{3}}
\put(30,30){\circle*{3}}
\put(50,10){\circle*{3}}
\put(50,30){\circle*{3}}

\put(20,20){\circle*{3}}

\put(10,10){\line(2,1){40}}
\put(10,10){\line(0,1){20}}
\put(10,30){\line(1,-1){20}}
\put(30,10){\line(0,1){20}}
\put(30,30){\line(1,-1){20}}
\put(50,10){\line(0,1){20}}

\put(9,31){\makebox(0,0)[br]{$y_i$}}
\put(9,9){\makebox(0,0)[tr]{$x_i$}}
\put(30,9){\makebox(0,0)[t]{$x_{i+1}$}}
\put(30,31){\makebox(0,0)[b]{$y_{i+1}$}}
\put(51,9){\makebox(0,0)[tl]{$x_{i+2}$}}
\put(51,31){\makebox(0,0)[bl]{$y_{i+2}$}}

\end{picture}
\begin{picture}(10,40)
\put(5,20){\makebox(0,0){or}}
\end{picture}
\begin{picture}(60,40)

\put(10,10){\circle*{3}}
\put(10,30){\circle*{3}}
\put(30,10){\circle*{3}}
\put(30,30){\circle*{3}}
\put(50,10){\circle*{3}}
\put(50,30){\circle*{3}}


\put(10,10){\line(2,1){40}}
\put(10,10){\line(0,1){20}}
\put(10,30){\line(1,-1){20}}
\put(30,10){\line(0,1){20}}
\put(30,30){\line(1,-1){20}}
\put(50,10){\line(0,1){20}}

\put(9,31){\makebox(0,0)[br]{$y_i$}}
\put(9,9){\makebox(0,0)[tr]{$x_i$}}
\put(30,9){\makebox(0,0)[t]{$x_{i+1}$}}
\put(30,31){\makebox(0,0)[b]{$y_{i+1}$}}
\put(51,9){\makebox(0,0)[tl]{$x_{i+2}$}}
\put(51,31){\makebox(0,0)[bl]{$y_{i+2}$}}

\end{picture}
$$
then it is not \qonered.
Later, we seek a 
sequence
$y_0$, $x_1$, \ldots, $y_{t-1}$, $x_t$ with
\condn\ 
such that
$\qe(x_0, y_0, \ldots, x_t, \infty)$ as large as possible,
where $\epsilon=1$ or $-1$.
If there is a part of the first kind in the \scn, we can replace it with
$y_0$, $x_1$, \ldots, $x_i$, $y_{i+2}$, $x_{i+2}$, \ldots, $x_t$
and obtain a sequence with larger $\qone(x_0,y_0,\ldots)$.
Further, if there is a part of the second kind in the \scn, we apply 
the replacement above and remove redundancy.
In fact, there is no redundancy of this kind is key to Lemma \ref{lem:nu up down}
and \S \ref{sec:conv polytope}.
\end{example}

Now we begin to analyze the property of generating system of $\omega^{(n)}$.
First we state the following (cf. \cite[Lemma 3.2]{mo}).

\begin{lemma}
\mylabel{lem:mo3.2}
Let $\nu$ be an element of $\TTTTTn(P)$.
If there is a possibly empty sequence 
$z_0$, $w_1$, \ldots, $z_{\ell-1}$, $w_\ell$ of elements of
$P\setminus\{x_0\}$ such that
$z_0>w_1<\cdots<z_{\ell-1}>w_\ell$ and
$\nu(w_i)-\nu(z_i)=\qndist(w_i,z_i)$
for any $0\leq i\leq \ell$, where we set $w_0=x_0$ and $z_\ell=\infty$,
then $\nu$ is a minimal element of $\TTTTTn(P)$.
\end{lemma}
The proof is almost identical with that of \cite[Lemma 3.2]{mo}.
Thus, we omit it.

Next we state a strong converse of this Lemma.

\begin{lemma}
\mylabel{lem:mo3.3}
Let $\nu$ be a minimal element of $\TTTTTn(P)$.
Then there is a possibly empty $\qn$-reduced sequence
$y_0$, $x_1$, \ldots, $y_{t-1}$, $x_t$ with \condn\ such that
\begin{equation}
\nu(x_i)-\nu(y_i)=\qndist(x_i,y_i)
\quad\mbox{for $0\leq i\leq t$,}
\mylabel{eq:mo3.3}
\end{equation}
where we set $y_t=\infty$.
\end{lemma}
Since the proof is almost identical with \cite[Lemma 3.3]{mo},
we omit it.
Note that the \scn\ we constructed in the proof of \cite[Lemma 3.3]{mo}
is  $\qone$-reduced.
By Lemmas \ref{lem:mo3.2} and \ref{lem:mo3.3}, we see that $\nu\in\TTTTTn(P)$
is a minimal element of $\TTTTTn(P)$ if and only if there exists a \scn\ 
which satisfies equations \refeq{eq:mo3.3}.

Noting that there are only finitely many sequences with \condn, we make the following.

\begin{definition}\rm
\mylabel{def:q max}
We set
$\qn_0\define\qndist(x_0,\infty)$
and
$\qnmax\define\max\{\qn(x_0,y_0,\ldots,x_t,\infty)\mid
y_0$, $x_1$, \ldots, $y_{t-1}$, $x_t$ is a $\qn$-reduced
\scn$\}$.
\end{definition}
By 
the fact that 
$\nu(x_0)=\nu(x_0)-\nu(\infty)\geq\qndist(x_0,\infty)$, for $\nu\in\TTTTTn(P)$,
we see that 
$\qn_0\leq\nu(x_0)$ for any minimal element $\nu$ of $\TTTTTn(P)$.
Further, 
by the same way as the proof of \cite[Corollary 3.5]{mo},
we see that $\nu(x_0)\leq \qnmax$ for any minimal element $\nu$ of $\TTTTTn(P)$
by Lemma \ref{lem:mo3.3}.
Thus, 
$\qn_0\leq d\leq\qnmax$ is a necessary condition that there is a 
generator of $\omega^{(n)}$ with degree $d$.

We show that this is also a sufficient condition in the following of this section.

\begin{definition}\rm
\mylabel{def:mo3.6}
Let $y_0$, $x_1$, \ldots, $y_{t-1}$, $x_t$ be a $\qn$-reduced
\scn.
Set $y_t=\infty$.
We define
$$
\mu_{(y_0,x_1, \ldots, y_{t-1},x_t)}(x_i)\define\qn(x_i,y_i,\ldots, x_t,\infty)
$$
and
$$
\mu_{(y_0,x_1, \ldots, y_{t-1},x_t)}(y_i)\define\qn(x_i,y_i,\ldots, x_t,\infty)
-\qndist(x_i,y_i)
$$
for $0\leq i\leq t$.
We also define
$$
\nudown_{(y_0,x_1,\ldots, y_{t-1},x_t)}(z)
\define\max\{\mu_{(y_0,x_1,\ldots,y_{t-1},x_t)}(y_j)+\qndist(z,y_j)\mid
y_j\geq z\}
$$
and
$$
\nuup_{(y_0,x_1,\ldots, y_{t-1},x_t)}(z)
\define\min\{\mu_{(y_0,x_1,\ldots,y_{t-1},x_t)}(x_i)-\qndist(x_i,z)\mid
x_i\leq z\}
$$
\end{definition}
Note that the definition of $\nuup_{(y_0,x_1,\ldots,y_{t-1},x_t)}$ is different
from that of \cite[Definition 3.6]{mo}.
Here we define $\nudown_{(y_0,x_1,\ldots,y_{t-1},x_t)}$ and
$\nuup_{(y_0,x_1,\ldots,y_{t-1},x_t)}$ for $\qn$-reduced \scn\ only.
Next we state basic properties of 
$\nudown_{(y_0,x_1,\ldots,y_{t-1},x_t)}$ and
$\nuup_{(y_0,x_1,\ldots,y_{t-1},x_t)}$.

\begin{lemma}
\mylabel{lem:nu up down}
Let $y_0$, $x_1$, \ldots, $y_{t-1}$, $x_t$
be a \qnred\ \scn.
We denote
$\mu_{(y_0,x_1,\ldots,y_{t-1},x_t)}$,
$\nudown_{(y_0,x_1,\ldots,y_{t-1},x_t)}$ and
$\nuup_{(y_0,x_1,\ldots,y_{t-1},x_t)}$ by
$\mu$, $\nudown$ and $\nuup$ respectively.
Then
$\nudown$, $\nuup$ are minimal elements of $\TTTTTn(P)$,
$\nudown(x_i)=
\nuup(x_i)=
\mu(x_i)$
and
$\nudown(y_i)=
\nuup(y_i)=
\mu(y_i)$
for $0\leq i\leq t$.
\end{lemma}
\begin{proof}
Suppose $z$, $z'\in P^+$ and $z\covered z'$.
Then it is easily verified that
\begin{equation}
\nudown(z)-\nudown(z')\geq n
\quad\mbox{and}\quad
\nuup(z)-\nuup(z')\geq n.
\mylabel{eq:enough diff}
\end{equation}

Next we show that $\nudown(x_i)=\mu(x_i)$ for $0\leq i\leq t$.
Take $j$ with $y_j\geq x_i$ and $\nudown(x_i)=\qndist(x_i,y_j)+\mu(y_j)$.
Then $j\geq i-1$ since $y_0$, $x_1$, \ldots, $y_{t-1}$, $x_t$ satisfies \condn.
If $j>i$, then since $y_0$, $x_1$, \ldots, $y_{t-1}$, $x_t$ is \qnred, we see that
\begin{eqnarray}
\nudown(x_i)&=&
\qndist(x_i,y_j)+\mu(y_j)
\nonumber\\
&<&\qn(x_i,y_i,\ldots, x_j,y_j)+\mu(y_j)
\mylabel{eq:str ineq1}
\\
&=&\mu(x_i)
\nonumber
\\
&=&\qndist(x_i,y_i)+\mu(y_i).
\nonumber
\end{eqnarray}
This contradicts to the definition of $\nudown$.
Thus, $j=i$ or $i-1$.
Since $\qndist(x_i,y_i)+\mu(y_i)=\mu(x_i)$ and
$\qndist(x_i,y_{i-1})+\mu(y_{i-1})=\mu(x_i)$ if $i>0$, we see that
$\nudown(x_i)=\mu(x_i)$ by the definition of $\nudown$.
We also see that $\nuup(y_i)=\mu(y_i)$ for $0\leq i\leq t$
by the same way.

Next we show that $\nuup(x_i)=\mu(x_i)$ for $0\leq i\leq t$.
Take $j$ with $x_j\leq x_i$ and $\nuup(x_i)=\mu(x_j)-\qndist(x_j,x_i)$.
Since $x_j\leq x_i<y_{i-1}$ and $y_0$, $x_1$, \ldots, $y_{t-1}$, $x_t$ satisfies \condn,
we see that $j\leq i$.
If $j<i$, then, since $y_0$, $x_1$, \ldots, $y_{t-1}$, $x_t$ is \qnred,
we see that
\begin{eqnarray}
\nuup(x_i)&=&\mu(x_j)-\qndist(x_j,x_i)
\nonumber\\
&=&\mu(x_j)-\qndist(x_j,x_i)-\qndist(x_i,y_i)+\qndist(x_i,y_i)
\nonumber\\
&\geq&\mu(x_j)-\qndist(x_j,y_i)+\qndist(x_i,y_i)
\nonumber\\
&>&\mu(x_j)-\qn(x_j,y_j,\ldots,x_i,y_i)+\qndist(x_i,y_i)
\mylabel{eq:str ineq2}
\\
&=&\mu(y_i)+\qndist(x_i,y_i)
\nonumber\\
&=&\mu(x_i)
\nonumber\\
&=&\mu(x_i)-\qndist(x_i,x_i).
\nonumber
\end{eqnarray}
This contradicts to the definition of $\nuup$.
Thus we see that $j=i$ and $\nuup(x_i)=\mu(x_i)-\qndist(x_i,x_i)=\mu(x_i)$.
We also see that $\nudown(y_i)=\mu(y_i)$ for $0\leq i\leq t$ by the same way.
In particular, $\nudown(\infty)=\nuup(\infty)=\mu(\infty)=0$.
Therefore, we see that $\nudown$, $\nuup\in\TTTTTn(P)$
by inequalities \refeq{eq:enough diff}.
By Lemma \ref{lem:mo3.2}, we see that $\nudown$ and $\nuup$ are minimal elements of $\TTTTTn(P)$.
\end{proof}
Next we note the following.

\begin{lemma}
\mylabel{lem:mo3.11}
Let $\nu$ be a minimal element of $\TTTTTn(P)$ and $k$ a positive integer.
Set
$$
\nu'(z)=\max\{\nu(z)-k,\qndist(z,\infty)\}
$$
for $z\in P^+$.
Then $\nu'$ is also a minimal element of $\TTTTTn(P)$.
\end{lemma}
\begin{proof}
It is easily verified that $\nu'\in\TTTTTn(P)$.
The rest is proved along the same line with 
\cite[Lemma 3.11]{mo}.
\end{proof}

Now we state the following.

\begin{thm}
\mylabel{thm:deg range}
There exists a generator of $\omega^{(n)}$ with degree $d$ if and
only if $\qn_0\leq d\leq\qnmax$.
\end{thm}
\begin{proof}
``Only if'' part is already proved after Definition \ref{def:q max}.

Let $d$ be an integer with $\qn_0\leq d\leq\qnmax$.
Set $k=\qnmax-d$ and take a \qnred\ sequence $y_0$, $x_1$, \ldots, $y_{t-1}$, $x_t$
with \condn\ such that
$\qn(x_0,y_0,\ldots, x_t,\infty)=\qnmax$.
Then $\nudown_{(y_0,x_1,\ldots, y_{t-1},x_t)}$
is a minimal element of $\TTTTTn(P)$ with
$\nudown_{(y_0,x_1,\ldots, y_{t-1},x_t)}(x_0)=
\mu_{(y_0,x_1,\ldots, y_{t-1},x_t)}(x_0)=
\qn(x_0,y_0,\ldots, x_t,\infty)=\qnmax$
by Lemma \ref{lem:nu up down}
and $\nu'\colon P^+\to\ZZZ$, $z\mapsto 
\max\{\nudown_{(y_0,x_1,\ldots, y_{t-1},x_t)}(z)-k,\qndist(z,\infty)\}$
is a minimal element of $\TTTTTn(P)$ by Lemma \ref{lem:mo3.11}.
Since $\nu'(x_0)=\qnmax-k=d$, we see that $T^{\nu'}$ is a generator
of $\omega^{(n)}$ with degree $d$.
\end{proof}

For any nonempty \qnred\ sequence $y_0$, $x_1$, \ldots, $y_{t-1}$, $x_t$
with \condn,
$\qn_0=\qndist(x_0,\infty)<\qn(x_0,y_0,\ldots, x_t,\infty)$.
Therefore, we obtain the following result from Theorem \ref{thm:deg range}.

\begin{thm}
\mylabel{thm:pm level cri}
$\rkh$ is level (resp.\ anticanonical level)
if and only if \qonered\ (resp. \qmonered)
\scn\ is the empty sequence only.
\end{thm}
Note that for level case, Theorem \ref{thm:pm level cri} is another expression of
\cite[Theorem 3.9]{mo} using the notion of \qonered\ \scn.

As a corollary, we see that the anticanonical counterpart of
\cite[Theorem 3.3]{ml}
(see also \cite[Corollary 3.10]{mo}) 
also holds.

\begin{cor}
\mylabel{cor:ml3.3}
If $\{z\in P\mid z\geq w\}$ is pure
for any $w\in P\setminus\{x_0\}$,
then $\rkh$ is level and anticanonical level.
\end{cor}
\begin{proof}
Suppose that there exists a \qmonered\ sequence $y_0$, $x_1$, \ldots, $y_{t-1}$, $x_t$ 
\condn\ with $t>0$.
Set $y_t\define\infty$.
Then
\begin{eqnarray*}
&&\qmone(x_0,y_0,\ldots, x_t,y_t)\\
&=&\sum_{i=0}^t\qmone\dist(x_i,y_i)-\sum_{i=1}^{t}\qmone\dist(x_{i},y_{i-1})\\
&=&\qmone\dist(x_0,y_0)+\sum_{i=1}^t(\qmone\dist(x_i,\infty)-\qmone\dist(y_i,\infty))\\
&&\qquad-\sum_{i=1}^t(\qmone\dist(x_i,\infty)-\qmone\dist(y_{i-1},\infty))\\
&=&\qmone\dist(x_0,y_0)+\qmone\dist(y_0,\infty)-\qmone\dist(\infty,\infty)\\
&\leq&\qmone\dist(x_0,\infty),
\end{eqnarray*}
contradicting the fact that $y_0$, $x_1$, \ldots, $y_{t-1}$, $x_t$ is \qmonered.
Thus, there is no \qmonered\ \scn\ except the empty sequence and we see by 
Theorem \ref{thm:pm level cri} that $\rkh$ is anticanonical 
level.

The 
level property is proved by the same way.
\end{proof}
%


\section{Convex polytope associated to a \qered\ \scn}

\mylabel{sec:conv polytope}

Let $\epsilon$ be $\pm1$.
In this section, we construct a convex polytope associated to a \qered\ \scn.
Fix a \qered\ sequence $y_0$, $x_1$, \ldots, $y_{t-1}$, $x_t$ with \condn.
We set $y_t\define\infty$.
We define a convex polytope from $y_0$, $x_1$, \ldots, $y_{t-1}$, $x_t$
and study the Ehrhart ring defined by this polytope.

Here we establish the notation of the Ehrhart ring.
Let $W$ be a finite set and $C$ an rational convex polytope in $\RRR^W$,
i.e., a convex polytope whose vertices are in $\QQQ^W$.
Also let $\KKK$ be a field, $\{X_w\}_{w\in W}$ a family of indeterminates 
indexed by $W$ and $Y$ an indeterminate.
Then the Ehrhart ring $\KKK[C]$ of $C$ over $\KKK$ in $\KKK[X_w^{\pm1}\mid w\in W][Y]$
is the subring of $\KKK[X_w^{\pm1}\mid w\in W][Y]$ generated by
$\{\prod_{w\in W} X_w^{\nu(w)}T^n\mid n\in \NNN$, $\nu\in nC\cap\ZZZ^W\}$ over $\KKK$.
It is known that $\dim\KKK[C]=\dim C+1$, see e.g.
\cite[(14.C) Theorem 23]{mat}.

Let $n$ be  a positive integer. 
Note that $y_0$, $x_1$, \ldots, $y_{t-1}$, $x_t$ is a $q^{(n\epsilon)}$-reduced
\scn.
We set 
$$
C^{(n\epsilon)}\define
\left\{
\nu\colon P^+\to\RRR\ \bigg|\  \vcenter{\hsize=.5\textwidth\relax\noindent
$\nu(\infty)=0$,
$\nu(z)-\nu(z')\geq n\epsilon$ for any $z$, $z'\in P^+$ with
$z\covered z'$ and $\nu(x_i)-\nu(y_i)=\qdist{n\epsilon}(x_i,y_i)$
for $0\leq i\leq t$.
}\right\}.
$$
For any $\nu\in C^{(n\epsilon)}$ and $z\in P$,
$\nu(z)\geq\qne(z,\infty)$ and $\nu(z)
=\nu(x_0)-(\nu(x_0)-\nu(z))
\leq \qne_{\max}-\qne(x_0,z)$ 
by the  argument after Definition \ref{def:q max}.
Thus, $C^{(n\epsilon)}$ is bounded, i.e., $C^{(n\epsilon)}$ is a convex polytope.
Since $\qdist{n\epsilon}(x_i,y_i)=n\qdist{\epsilon}(x_i,y_i)$ for $0\leq i\leq t$,
we see that $C^{(n\epsilon)}=nC^{(\epsilon)}$.
Further, if $n\geq 2$ and $\nu\in C^{(n\epsilon)}\cap\ZZZ^{P}$,
then it is easily
verified that $\lfloor\frac{\nu}{n}\rfloor\in C^{(\epsilon)}\cap\ZZZ^{P}$
and $\nu-\lfloor\frac{\nu}{n}\rfloor\in C^{((n-1)\epsilon)}\cap\ZZZ^{P}$.
Thus, we see by induction that any element $\nu\in C^{(n\epsilon)}\cap\ZZZ^{P}$ can 
be written as a sum of $n$ elements of $C^{(\epsilon)}\cap\ZZZ^{P}$.
Since $C^{(n\epsilon)}=nC^{(\epsilon)}$, we see that the Ehrhart ring defined
by $C^{(\epsilon)}$ is a standard graded ring, i.e., 
generated by 
the degree 1 part
over the base field.
In particular, $C^{(\epsilon)}$ is integral, i.e., all the vertices 
of $C^{(\epsilon)}$ are contained
in $\ZZZ^P$.

Next we consider the dimension of $C^{(\epsilon)}$.
Set $G_i=\{z\in[x_i,y_i]_{P^+}\mid\qedist(x_i,z)+\qedist(z,y_i)=\qedist(x_i,y_i)\}$
for $0\leq i\leq t$ and $G=G_0\cup\cdots\cup G_t$.
Note that $x_i$, $y_i\in G_i$ for $0\leq i\leq t$.
Further for any $\nu\in C^{(\epsilon)}$ and $z\in G_i$,
\begin{equation}
\mylabel{eq:nuz fixed}
\nu(z)=\nu(y_i)+\qedist(z,y_i).
\end{equation}
since $\nu(z)-\nu(y_i)\geq\qedist(z,y_i)$, $\nu(x_i)-\nu(z)\geq\qedist(x_i,z)$,
$\qedist(x_i,z)+\qedist(z,y_i)\leq \qedist(x_i,y_i)$
and $\nu(x_i)-\nu(y_i)=\qedist(x_i,y_i)$.
Therefore,
\begin{equation}
\dim C^{(\epsilon)}\leq
\#(P\setminus G)+t.
\mylabel{eq:dim c ub}
\end{equation}

We show the converse inequality by showing that there are affinely independent
elements of $C^{(\epsilon)}$ consisting of 
$\#(P\setminus G)+t+1$ elements.
First we make the following.

\begin{definition}\rm
\mylabel{def:mu s nu s}
Let $s$ be an integer with $0\leq s\leq t$.
We set $\mu=\mu_{(y_0,x_1,\ldots, y_{t-1}, x_t)}$,
$$
\mu_s(x_i)\define
\left\{
\begin{array}{ll}
\mu(x_i)\quad&\mbox{if $i\geq s$}\\
\mu(x_i)-1\quad&\mbox{if $i< s$}
\end{array}\right.
$$
and
$$
\mu_s(y_i)\define
\left\{\begin{array}{ll}
\mu(y_i)\quad&\mbox{if $i\geq s$}\\
\mu(y_i)-1\quad&\mbox{if $i< s$.}
\end{array}\right.
$$
Further, we define maps $\nudown_s$ and $\nuup_s$ from $P^+$ to $\ZZZ$ by
$$
\nudown_s(z)\define\max\{\qedist(z,y_i)+\mu_s(y_i)\mid y_i\geq z\}
$$
and
$$
\nuup_s(z)\define\min\{\mu_s(x_i)-\qedist(x_i,z)\mid x_i\leq z\}.
$$
\end{definition}
Note that it is easily verified that
$
\nudown_s(z)-\nudown_s(z')\geq\epsilon$
and
$\nuup_s(z)-\nuup_s(z')\geq\epsilon$
for any $z$, $z'\in P^+$ with $z\covered z'$.
Further, since the inequalities \refeq{eq:str ineq1} and \refeq{eq:str ineq2}
are strict, we can show that
\begin{equation}
\nudown_s(x_i)=\nuup_s(x_i)=\mu_s(x_i)
\mylabel{eq:value xi}
\end{equation}
and
\begin{equation}
\nudown_s(y_i)=\nuup_s(y_i)=\mu_s(y_i)
\mylabel{eq:value yi}
\end{equation}
for any $i$ with $0\leq i\leq t$ by the same way as the proof of 
Lemma \ref{lem:nu up down}.
In particular, $\nudown_s(\infty)=\nuup_s(\infty)=0$.
Therefore, we see that $\nudown_s$, $\nuup_s\in\TTTTT^{(\epsilon)}(P)$
for any $s$ with $0\leq s\leq t$.

Since $\nudown_s(x_i)=\nuup_s(x_i)=\mu_s(x_i)$ and 
$\nudown_s(y_i)=\nuup_s(y_i)=\mu_s(y_i)$,
we see that
$\nudown_s(x_i)-\nudown_s(y_i)=\qedist(x_i,y_i)$
and
$\nuup_s(x_i)-\nuup_s(y_i)=\qedist(x_i,y_i)$
for any $0\leq i\leq t$.
Therefore, $\nudown_s$, $\nuup_s\in C^{(\epsilon)}$ for any $0\leq s\leq t$.
Note that $\mu_0=\mu$,
$\nudown_0=\nudown_{(y_0,x_1,\ldots, y_{t-1},x_t)}$
and
$\nuup_0=\nuup_{(y_0,x_1,\ldots, y_{t-1},x_t)}$.

Next we state the following.

\begin{lemma}
\mylabel{lem:up geq down}
Let $s$ be an integer with $0\leq s\leq t$.
Then $\nudown_s(z)\leq\nuup_s(z)$ for any $z\in P^+$.
\end{lemma}
\begin{proof}
Take $i$ and $j$ such that $x_i\leq z$, 
$\nuup_s(z)=\mu_s(x_i)-\qedist(x_i,z)$ and
$y_j\geq z$, $\nudown_s(z)=\mu_s(y_j)+\qedist(z,y_j)$.
Then $x_i\leq y_j$.
Therefore $j\geq i-1$ since $y_0$, $x_1$, \ldots, $y_{t-1}$, $x_t$
satisfies \condn.
Further,
\begin{eqnarray*}
\nuup_s(z)-\nudown_s(z)
&=&\mu_s(x_i)-\qedist(x_i,z)-\qedist(z,y_j)-\mu_s(y_j)\\
&\geq&\mu_s(x_i)-\mu_s(y_j)-\qedist(x_i,y_j).
\end{eqnarray*}
If $j=i-1$ or $j=i$, then
$$
\mu_s(x_i)-\mu_s(y_j)-\qedist(x_i,y_j)\geq\mu(x_i)-\mu(y_j)-\qedist(x_i,y_j)=0.
$$
If $j>i$, then since 
$y_0$, $x_1$, \ldots, $y_{t-1}$, $x_t$
is \qered,
we see that
\begin{eqnarray*}
&&\mu_s(x_i)-\mu_s(y_j)-\qedist(x_i,y_j)\\
&\geq&\mu(x_i)-\mu(y_j)-1-\qedist(x_i,y_j)\\
&=&\qe(x_i,y_i,\ldots,x_j,y_j)-1-\qedist(x_i,y_j)\\
&\geq&0.
\end{eqnarray*}
Thus, $\nuup_s(z)-\nudown_s(z)\geq 0$.
\end{proof}

We set $\nu_{00}\define\nudown_t$ and
$F_0\define\{z\in P\mid \nu_{00}(z)<\nuup_t(z)\}$.
Further, we set $F_0=\{z_{01}$, $z_{02}$, \ldots, $z_{0k(0)}\}$ so that
$z_{01}$, $z_{02}$, \ldots, $z_{0k(0)}$ is a linear extension of 
$F_0$, i.e., if $z_{0i}<z_{0j}$ then $i<j$.

For $j$ with $1\leq j\leq k(0)$, set $F_{0j}\define\{z_{01}$, \ldots, $z_{0j}\}$ and
for $z\in P^+$
$$
\nu_{0j}(z)=
\left\{\begin{array}{ll}
\nu_{00}(z)\quad&\mbox{if $z\not\in F_{0j}$,}\\
\nu_{00}(z)+1\quad&\mbox{if $z\in F_{0j}$.}
\end{array}
\right.
$$
Then for any $1\leq j\leq k(0)$, the following fact holds.

\begin{lemma}
\mylabel{lem:nu 0j in c}
Let $j$ be an integer with $1\leq j\leq k(0)$.
Then $\nu_{0j}$ is an element of $C^{(\epsilon)}$ such that
$\nu_{0j}(x_i)=\mu_t(x_i)$ and 
$\nu_{0j}(y_i)=\mu_t(y_i)$ for any $0\leq i\leq t$.
\end{lemma}
\begin{proof}
Suppose $z$, $z'\in P^+$ and $z\covered z'$.
If $z'\not\in F_{0j}$ or $z\in F_{0j}$, then 
$\nu_{0j}(z)-\nu_{0j}(z')\geq\nu_{00}(z)-\nu_{00}(z')\geq\epsilon$.
Assume that $z\not\in F_{0j}$ and $z'\in F_{0j}$.
Since $z_{01}$, \ldots, $z_{0k(0)}$ is a linear extension of $F_0$,
we see that $z\not\in F_0$.
Therefore, $\nuup_t(z)=\nu_{00}(z)$.
On the other hand, since $z'\in F_0$, we see that
$\nuup_t(z')\geq\nu_{00}(z')+1$.
Thus,
$$
\nu_{0j}(z)-\nu_{0j}(z')=\nu_{00}(z)-(\nu_{00}(z')+1)
\geq\nuup_t(z)-\nuup_t(z')\geq\epsilon.
$$
Since $\nuup_t(x_i)=\nudown_t(x_i)=\mu_t(x_i)$
and $\nuup_t(y_i)=\nudown_t(y_i)=\mu_t(y_i)$,
we see that $x_i$, $y_i\not\in F_0$ for any $0\leq i\leq t$.
Thus, $\nu_{0j}(x_i)=\mu_t(x_i)$ and $\nu_{0j}(y_i)=\mu_t(y_i)$
for any $0\leq i\leq t$.
In particular,
$\nu_{0j}(\infty)=\nu_{0j}(y_t)=0$
and 
$\nu_{0j}(x_i)-\nu_{0j}(y_i)=\mu_t(x_i)-\mu_t(y_i)=\qedist(x_i,y_i)$.
Thus, we see that $\nu_{0j}\in C^{(\epsilon)}$.
\end{proof}

By the above lemma, we see that $\nu_{0k(0)}\in \TTTTTe(P)$.
We set  $\nu_{10}\define\max\{\nudown_{t-1}, \nu_{0k(0)}\}$.
Then we see the following.

\begin{lemma}
\mylabel{lem:nu 10 leq up}
$\nu_{10}$ is an element of $C^{(\epsilon)}$ such that
$\nu_{10}(x_i)=\mu_{t-1}(x_i)$,
$\nu_{10}(y_i)=\mu_{t-1}(y_i)$
for $0\leq i\leq t$ and $\nu_{10}(z)\leq\nuup_{t-1}(z)$
for any $z\in P^+$.
\end{lemma}
\begin{proof}
The first part of the assertion follows from
Lemmas \ref{lem:op nu} and \ref{lem:nu 0j in c},
equalities \refeq{eq:value xi}, \refeq{eq:value yi}
and the definitions of $\mu_{t-1}$ and $\mu_t$.

Let $z$ be an arbitrary element of $P^+$.
If $\nu_{10}(z)=\nudown_{t-1}(z)$, then by Lemma \ref{lem:up geq down},
we see that $\nu_{10}(z)\leq \nuup_{t-1}(z)$.
Suppose that $\nu_{10}(z)=\nu_{0k(0)}(z)$.
If $z\not\in F_0$, then 
$$
\nu_{0k(0)}(z)=\nudown_t(z)=\nuup_t(z)\leq\nuup_{t-1}(z)
$$
by the definition of $\nuup_t$ and $\nuup_{t-1}$.
If $z\in F_0$, then $\nu_{00}(z)<\nuup_t(z)$
by the definition of $F_0$.
Therefore, 
$$
\nu_{0k(0)}(z)=\nu_{00}(z)+1\leq\nuup_t(z)\leq\nuup_{t-1}(z).
$$
\end{proof}

We set $F_1\define\{z\in P\mid \nu_{10}(z)<\nuup_{t-1}(z)\}
\setminus F_0$ and set
$F_1=\{z_{11}, \ldots, z_{1k(1)}\}$
so that $z_{11}$, \ldots, $z_{1k(1)}$ is a linear extension of $F_1$.
We also set $F_{1j}\define\{z_{11}$, \ldots, $z_{1j}\}$ and
for $z\in P^+$,
$$
\nu_{1j}(z)=
\left\{\begin{array}{ll}
\nu_{10}(z)\quad&\mbox{if $z\not\in F_{1j}$}\\
\nu_{10}(z)+1\quad&\mbox{if $z\in F_{1j}$}
\end{array}
\right.
$$
for $1\leq j\leq k(1)$.
Then by the same argument as the proof of Lemma \ref{lem:nu 0j in c},
we see that for any $1\leq j\leq k(1)$,
$\nu_{1j}$ is an element of $C^{(\epsilon)}$ with
$\nu_{1j}(x_i)=\mu_{t-1}(x_i)$ and
$\nu_{1j}(y_i)=\mu_{t-1}(y_i)$ for any $0\leq i\leq t$.

Set $\nu_{20}\define\max\{\nu_{1k(1)}, \nudown_{t-2}\}$.
Then by the same argument as the proof of Lemma \ref{lem:nu 10 leq up},
we see that $\nu_{20}$ is an element of $C^{(\epsilon)}$
such that $\nu_{20}(x_i)=\mu_{t-2}(x_i)$ and
$\nu_{20}(y_i)=\mu_{t-2}(y_i)$ for $0\leq i\leq t$
and $\nu_{20}(z)\leq\nuup_{t-2}(z)$ for any $z\in P^+$.
Thus, we can repeat this argument by setting
$F_2\define\{z\in P\mid \nu_{20}(z)<\nuup_{t-2}(z)\}
\setminus(F_0\cup F_1)$
and taking a linear extension 
$F_2=\{z_{21},\ldots, z_{2k(2)}\}$, setting
$F_{2j}\define\{z_{21},\ldots, z_{2j}\}$
and for $z\in P^+$,
$$
\nu_{2j}(z)=
\left\{\begin{array}{ll}
\nu_{20}(z)\quad&\mbox{if $z\not\in F_{2j}$}\\
\nu_{20}(z)+1\quad&\mbox{if $z\in F_{2j}$}
\end{array}
\right.
$$
for $1\leq j\leq k(2)$ and so on.

Finally, we define $k(0)+k(1)+\cdots+k(t)+t+1$ elements 
$$
\nu_{00}, \nu_{01}, \ldots, \nu_{0k(0)},
\nu_{10}, \nu_{11}, \ldots, \nu_{1k(1)},
\ldots,
\nu_{t0}, \nu_{t1}, \ldots, \nu_{tk(t)}
$$
of $C^{(\epsilon)}$.
Since 
$$
\nu_{ij}(z)-\nu_{i,j-1}(z)=
\left\{\begin{array}{ll}
1&\quad\mbox{if $z=z_{ij}$}\\
0&\quad\mbox{otherwise}
\end{array}\right.
$$
for any $0\leq i\leq t$ and $1\leq j\leq k(i)$
and
$$
\nu_{i0}(y_j)-\nu_{i-1,k(i-1)}(y_j)=
\left\{\begin{array}{ll}
1\quad\mbox{if $j=t-i$}\\
0\quad\mbox{if $j\neq t-i$}\\
\end{array}\right.
$$
for any $1\leq i\leq t$,
we see that 
$$
\nu_{00}, \nu_{01}, \ldots, \nu_{0k(0)},
\nu_{10}, \nu_{11}, \ldots, \nu_{1k(1)},
\ldots,
\nu_{t0}, \nu_{t1}, \ldots, \nu_{tk(t)}
$$
are affinely independent,
since $F_i\cap F_j=\emptyset$ for $i\neq j$ and 
$y_j\not\in F_i$ for any $i$ and $j$.
Since $\nu_{ij}\in C^{(\epsilon)}$ for any $0\leq i\leq t$
and $0\leq j\leq k(i)$, we see that
\begin{equation}
\dim C^{(\epsilon)}\geq
k(0)+k(1)+\cdots+k(t)+t.
\mylabel{eq:dim c lb}
\end{equation}

Next we set $F=F_0\cup F_1\cup\cdots\cup F_t$ and  state the following.

\begin{lemma}
\mylabel{lem:in g char}
For $w\in P$ the following conditions are equivalent.
\begin{enumerate}
\item
\mylabel{item:nin f}
$w\not\in F$.
\item
\mylabel{item:up down eq}
$\nudown_s(w)=\nuup_s(w)$ for any $s$ with $0\leq s\leq t$.
\item
\mylabel{item:in g}
$w\in G$.
\end{enumerate}
\end{lemma}
\begin{proof}
\ref{item:nin f}$\Rightarrow$\ref{item:up down eq}:
Since $\nu_{00}=\nudown_t$ and $w\not\in F_0
=\{z\in P\mid \nu_{00}(z)<\nuup_t(z)\}$,
we see that $\nudown_t(w)=\nuup_t(w)$.
Further, $\nu_{0k(0)}(w)=\nu_{00}(w)=\nudown_t(w)$,
since $w\not\in F_0$.
Therefore, $\nu_{10}(w)=\max\{\nudown_{t-1}(w),\nudown_t(w)\}
=\nudown_{t-1}(w)$.

Since $w\not\in F_1=\{z\in P\mid \nu_{10}(z)<\nuup_{t-1}(z)\}$,
we see that $\nudown_{t-1}(w)=\nu_{10}(w)=\nuup_{t-1}(w)$.
Further, $\nu_{1k(1)}(w)=\nu_{10}(w)=\nudown_{t-1}(w)$,
since $w\not\in F_1$.
By repeating this argument, we see \ref{item:up down eq}.

\ref{item:up down eq}$\Rightarrow$\ref{item:in g}:
By assumption, we see that $\nudown_0(w)=\nuup_0(w)$.
By the definition of $\nudown_0$ and $\nuup_0$, we see that 
there are $i$ and $j$ such that
$x_i\leq w\leq y_j$, $\nudown_0(w)=\mu_0(y_j)+\qedist(w,y_j)$
and $\nuup_0(w)=\mu_0(x_i)-\qedist(x_i,w)$.
Take $j$ maximal and $i$ minimal.
We shall show that $i=j$.
Since $x_i\leq y_j$ and
$y_0$, $x_1$, \ldots, $y_{t-1}$, $x_t$
satisfies \condn, we see that $j\geq i-1$.

First suppose that $j\geq i+1$.
Then
\begin{eqnarray*}
\nuup_0(w)-\nudown_0(w)&=&
\mu_0(x_i)-\qedist(x_i,w)-\qedist(w,y_j)-\mu_0(y_j)\\
&\geq&\mu(x_i)-\mu(y_j)-\qedist(x_i,y_j)\\
&=&\qe(x_i,y_i\ldots,x_j,y_j)-\qedist(x_i,y_j)\\
&>&0,
\end{eqnarray*}
since $y_0$, $x_1$, \ldots, $y_{t-1}$, $x_t$ is \qered.
This contradicts to the assumption.

Next suppose that $j=i-1$.
Take $x_\ell$ and $y_{\ell'}$
such that $x_{\ell}\leq w\leq y_{\ell'}$,
$\nudown_i(w)=\mu_i(y_{\ell'})+\qedist(w,y_{\ell'})$
and
$\nuup_i(w)=\mu_i(x_{\ell})-\qedist(x_\ell,w)$.
If $\ell'\geq i=j+1$, then
\begin{eqnarray*}
\nudown_i(w)&=&
\mu_i(y_{\ell'})+\qedist(w,y_{\ell'})\\
&=&\mu_0(y_{\ell'})+\qedist(w,y_{\ell'})\\
&<&\nudown_0(w),
\end{eqnarray*}
since we took $j$ maximal.
If $\ell'<i$, then 
\begin{eqnarray*}
\nudown_i(w)&=&
\mu_i(y_{\ell'})+\qedist(w,y_{\ell'})\\
&=&\mu_0(y_{\ell'})-1+\qedist(w,y_{\ell'})\\
&<&\nudown_0(w).
\end{eqnarray*}

On the other hand,
if $\ell<i$, 
\begin{eqnarray*}
\nuup_i(w)&=&
\mu_i(x_\ell)-\qedist(x_\ell,w)\\
&=&\mu_0(x_\ell)-1-\qedist(x_\ell,w)\\
&\geq&\nuup_0(w),
\end{eqnarray*}
since we took $i$ minimal.
If $\ell\geq i$, then
\begin{eqnarray*}
\nuup_i(w)&=&
\mu_i(x_\ell)-\qedist(x_\ell,w)\\
&=&\mu_0(x_\ell)-\qedist(x_\ell,w)\\
&\geq&\nuup_0(w).
\end{eqnarray*}
Thus, we see that
$$
\nudown_i(w)<\nudown_0(w)=\nuup_0(w)\leq\nuup_i(w).
$$
This contradicts to the assumption.

Therefore, $j=i$ and we see that
\begin{eqnarray*}
0&=&\nuup_0(w)-\nudown_0(w)\\
&=&\mu(x_i)-\qedist(x_i,w)-\qedist(w,y_i)-\mu(y_i)\\
&=&\qedist(x_i,y_i)-\qedist(x_i,w)-\qedist(w,y_i).
\end{eqnarray*}
This means $w\in G_i$.

\ref{item:in g}$\Rightarrow$\ref{item:nin f}:
Suppose that $w\in G_i$.
Let $s$ be an arbitrary integer with $0\leq s\leq t$.
Since $\nuup_s(w)\leq\mu_s(x_i)-\qedist(x_i,w)$ and
$\nudown_s(w)\geq\mu_s(y_i)+\qedist(w,y_i)$
by the definition of $\nudown_s$ and $\nuup_s$, we see that
\begin{eqnarray*}
\nuup_s(w)-\nudown_s(w)&\leq&
\mu_s(x_i)-\qedist(x_i,w)-\qedist(w,y_i)-\mu_s(y_i)\\
&=&\mu_s(x_i)-\mu_s(y_i)-\qedist(x_i,y_i)\\
&=&0
\end{eqnarray*}
since $w\in G_i$.
Since
$\nudown_s(w)\leq\nu_{t-s,0}(w)\leq\nuup_s(w)$,
we see that
$\nu_{t-s,0}(w)=\nuup_s(w)$.
Therefore, $w\not\in F_{t-s}$.
Since $s$ is an arbitrary integer with $0\leq s\leq t$,
we see that 
$w\not\in F$.
\end{proof}

By the above lemma, we see that $P\setminus G=F$ and therefore
$$
\#(P\setminus G)=\#F=k(0)+k(1)+\cdots+k(t),
$$
since 
$F_i\cap F_j=\emptyset$ if $i\neq j$.
Therefore, by inequalities \refeq{eq:dim c ub}  and \refeq{eq:dim c lb}, we see the following.
\begin{thm}
\mylabel{thm:dim ce}
$\dim C^{(\epsilon)}=\#(P\setminus G)+t$.
\end{thm}

\begin{remark}\rm
\mylabel{rem:ce full}
By equation \refeq{eq:nuz fixed},
$\nu(z)=\nu(y_i)+\qedist(z,y_i)$ for any $\nu\in C^{(\epsilon)}$ and $z\in G_i$.
Therefore, by Theorem \ref{thm:dim ce}, we see that $C^{(\epsilon)}$ is essentially
a full dimensional convex polytope in $\RRR^{(P\setminus G)\cup\{y_0,\ldots,y_{t-1}\}}$.
\end{remark}

By considering the case where the sequence under consideration in this section is the 
empty sequence, we see by Theorem \ref{thm:dim ce}, the following.

\begin{cor}
\mylabel{cor:non min max}
If $t=0$, then
$\dim C^{(1)}=\#\pnonmax$, where $\pnonmax\define\{z\in P\mid z$ is not in any 
chain of $P$ of maximal length$\}$ (resp.\ 
$\dim C^{(-1)}=\#\pnonmin$, 
where $\pnonmin=\{z\in P\mid z$ is not in any maximal chain of $P$ of minimal length$\}$).
\end{cor}


\section{Canonical and anticanonical analytic spreads}

\mylabel{sec:anal spread}

In this section, we describe 
the fiber cones $\bigoplus_{n\geq0}\omega^n/\mmmm\omega^n$ and
$\bigoplus_{n\geq0}(\omega^{(-1)})^n/\mmmm(\omega^{(-1)})^n$
and analytic spreads 
$\dim\bigoplus_{n\geq0}\omega^n/\mmmm\omega^n$ and
$\dim\bigoplus_{n\geq0}(\omega^{(-1)})^n/\mmmm(\omega^{(-1)})^n$
of the canonical and anticanonical ideals of the Hibi ring $\rkh$
in terms of the notation introduced in the previous section,
where $\mmmm$ is the irrelevant maximal ideal of $\rkh$.
By Theorem \ref{thm:omega n basis}, we see that
$\omega^n=\omega^{(n)}$ and $(\omega^{(-1)})^n=\omega^{(-n)}$
for a positive integer $n$.
Therefore, we consider the ring
$$
\bigoplus_{n\geq 0}\omega^{(n\epsilon)}/\mmmm\omega^{(n\epsilon)},
$$
where $\epsilon=\pm1$.

Set $N^{(\epsilon)}\define\{(y_0,x_1,\ldots, y_{t-1},x_t)\mid
y_0$, $x_1$, \ldots, $y_{t-1}$, $x_t$ is a \qered\ \scn$\}$.
For $(y_0, x_1, \ldots, y_{t-1}, x_t)\in N^{(\epsilon)}$, we 
denote by $C^{(\epsilon)}_{(y_0,x_1,\ldots, y_{t-1},x_t)}
\subset\RRR^P$
 the 
convex polytope defined by $y_0$, $x_1$, \ldots, $y_{t-1}$ , $x_t$
in the previous section and
$R^{(\epsilon)}_{(y_0,x_1,\ldots, y_{t-1},x_t)}
\subset\KKK[T^{\pm1}_x\mid x\in P\setminus\{x_0\}][T_{x_0}]$
 the Ehrhart ring defined
by $C^{(\epsilon)}_{(y_0,x_1,\ldots, y_{t-1},x_t)}$ 
over $\KKK$.
Further, we denote by $G_{i,(y_0,x_1,\ldots,y_{t-1},x_t)}^{(\epsilon)}$
and $G_{(y_0,x_1,\ldots,y_{t-1},x_t)}^{(\epsilon)}$ the sets denoted by
$G_i$ or $G$ in the previous section.
Then by Theorem \ref{thm:dim ce}, we see that
$\dim C^{(\epsilon)}_{(y_0,x_1,\ldots,y_{t-1},x_t)}=
\#(P\setminus G_{(y_0,x_1,\ldots, y_{t-1},x_t)})+t$
for any $(y_0,x_1,\ldots, y_{t-1},x_t)\in N^{(\epsilon)}$.
Moreover, for any positive integer $n$,
$nC^{(\epsilon)}_{(y_0,x_1,\ldots, y_{t-1},x_t)}
=C^{(n\epsilon)}_{(y_0,x_1,\ldots, y_{t-1},x_t)}$.
Further by Lemma \ref{lem:mo3.2},
we see that any $\nu\in C^{(n\epsilon)}_{(y_0,x_1,\ldots, y_{t-1},x_t)}\cap\ZZZ^{P}$
is a minimal element of $\TTTTT^{(n\epsilon)}(P)$ and therefore
$T^\nu$ is a generator of $\omega^{(n\epsilon)}$, i.e.,
the residue class of $T^\nu$ is
a basis element of $\omega^{(n\epsilon)}/\mmmm\omega^{(n\epsilon)}$.
Therefore, $R^{(\epsilon)}_{(y_0,x_1,\ldots,y_{t-1},x_t)}$ is
embedded in 
$\bigoplus_{n\geq 0}\omega^{(n\epsilon)}/\mmmm\omega^{(n\epsilon)}$.

Conversely, assume that 
$T^\nu$ is a generator of $\omega^{(n\epsilon)}$,
where $n$ is a positive integer.
Then by Lemma \ref{lem:mo3.3}, we see that there is a \qnered\ sequence
$y_0$, $x_1$, \ldots, $y_{t-1}$, $x_t$ with \condn\ such that
$\nu(x_i)-\nu(y_i)=\qnedist(x_i,y_i)$ for $0\leq i\leq t$, where we set
$y_t=\infty$.
Since a \scn\ is \qered\ if and only if \qnered,
we see that
$(y_0,x_1, \ldots, y_{t-1},x_t)\in N^{(\epsilon)}$.
Further, since $\nu\in\TTTTT^{(n\epsilon)}(P)$, we see that
$\nu\in C^{(n\epsilon)}_{(y_0,x_1,\ldots,y_{t-1},x_t)}\cap\ZZZ^{P}
=nC^{(\epsilon)}_{(y_0,x_1,\ldots,y_{t-1},x_t)}\cap\ZZZ^{P}$ and
we can consider 
that $T^\nu$ is an element of 
$R^{(\epsilon)}_{(y_0,x_1,\ldots,y_{t-1},x_t)}$
with degree $n$.
Thus, we see that
\begin{equation}
\mylabel{eq:fc ne}
\bigoplus_{n\geq 0}\omega^{(n\epsilon)}/\mmmm\omega^{(n\epsilon)}=
\sum_{(y_0,x_1,\ldots,y_{t-1},x_t)\in N^{(\epsilon)}}
R^{(\epsilon)}_{(y_0,x_1,\ldots,y_{t-1},x_t)}.
\end{equation}
Since there are only finitely many sequences with \condn,
we see,
 by considering the Hilbert function of
$\bigoplus_{n\geq0}\omega^{(n\epsilon)}/\mmmm\omega^{(n\epsilon)}$,
that
\begin{eqnarray*}
\dim(\bigoplus_{n\geq0}\omega^{(n\epsilon)}/\mmmm\omega^{(n\epsilon)})
&=&\max_{(y_0,x_1,\ldots, y_{t-1},x_t)\in N^{(\epsilon)}}
\dim 
R^{(\epsilon)}_{(y_0,x_1,\ldots,y_{t-1},x_t)}
\\
&=&\max_{(y_0,x_1,\ldots, y_{t-1},x_t)\in N^{(\epsilon)}}
\dim 
C^{(\epsilon)}_{(y_0,x_1,\ldots,y_{t-1},x_t)}+1\\
&=&\max_{(y_0,x_1,\ldots, y_{t-1},x_t)\in N^{(\epsilon)}}
\#(P\setminus 
G^{(\epsilon)}_{(y_0,x_1,\ldots,y_{t-1},x_t)})+t+1
\end{eqnarray*}
Therefore, we see the following.

\begin{thm}
\mylabel{thm:can anal spr}
The fiber cone of the canonical (resp.\ anticanonical) ideal of the Hibi ring
$\rkh$ is 
$$
\sum_{(y_0,x_1,\ldots,y_{t-1},x_t)\in N^{(1)}}
R^{(1)}_{(y_0,x_1,\ldots,y_{t-1},x_t)}
$$
(resp.\ 
$\sum_{(y_0,x_1,\ldots,y_{t-1},x_t)\in N^{(-1)}}
R^{(-1)}_{(y_0,x_1,\ldots,y_{t-1},x_t)}$)
and 
the canonical (resp.\ anticanonical) analytic spread 
is
$$
\max_{(y_0,x_1,\ldots, y_{t-1},x_t)\in N^{(1)}}
\dim 
C^{(1)}_{(y_0,x_1,\ldots,y_{t-1},x_t)}+1
$$
(resp.\
$
\max_{(y_0,x_1,\ldots, y_{t-1},x_t)\in N^{(-1)}}
\dim 
C^{(-1)}_{(y_0,x_1,\ldots,y_{t-1},x_t)}+1
$).
\end{thm}
As a special case, we see
by Theorems \ref{thm:pm level cri},
\ref{thm:can anal spr}
and Corollary \ref{cor:non min max},
the following fact whose anticanonical level part is 
\cite[Theorem 4.6]{pag}.

\begin{cor}
If $\rkh$ is level (resp.\ anticanonical level), then the 
canonical (resp.\ anticanonical) analytic spread of $\rkh$
is
$\#\pnonmax+1$
(resp.\
$\#\pnonmin+1$).
\end{cor}

\begin{example}\rm
\mylabel{ex:anal spread}
Let $P_1\setminus\{x_0\}$, $P_2\setminus\{x_0\}$ and $P_3\setminus\{x_0\}$ be
the poset with the following Hasse diagram respectively.
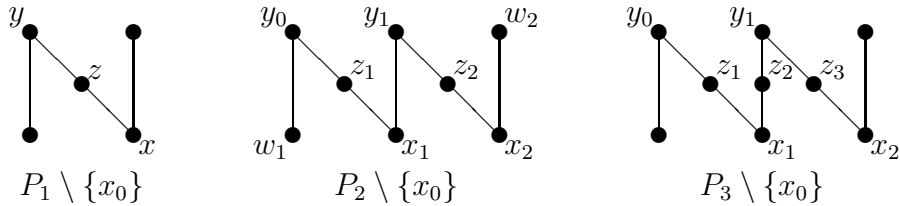
\unitlength=.005\textwidth\relax
$$
\vcenter{%
\begin{picture}(40,33)
\put(20,3){\makebox(0,0)[t]{$P_1\setminus\{x_0\}$}}

\put(10,10){\circle*{3}}

\put(10,30){\circle*{3}}
\put(20,20){\circle*{3}}
\put(30,10){\circle*{3}}
\put(30,30){\circle*{3}}

\put(10,10){\line(0,1){20}}
\put(10,30){\line(1,-1){20}}
\put(30,10){\line(0,1){20}}

\put(9,31){\makebox(0,0)[br]{$y$}}
\put(21,21){\makebox(0,0)[bl]{$z$}}
\put(31,9){\makebox(0,0)[tl]{$x$}}

\end{picture}
\hfil
\begin{picture}(60,33)
\put(30,3){\makebox(0,0)[t]{$P_2\setminus\{x_0\}$}}

\put(10,10){\circle*{3}}
\put(10,30){\circle*{3}}
\put(20,20){\circle*{3}}
\put(30,10){\circle*{3}}
\put(30,30){\circle*{3}}
\put(40,20){\circle*{3}}
\put(50,10){\circle*{3}}
\put(50,30){\circle*{3}}

\put(10,10){\line(0,1){20}}
\put(10,30){\line(1,-1){20}}
\put(30,10){\line(0,1){20}}
\put(30,30){\line(1,-1){20}}
\put(50,10){\line(0,1){20}}

\put(9,31){\makebox(0,0)[br]{$y_0$}}
\put(21,21){\makebox(0,0)[bl]{$z_1$}}
\put(31,9){\makebox(0,0)[tl]{$x_1$}}
\put(29,31){\makebox(0,0)[br]{$y_1$}}
\put(41,21){\makebox(0,0)[bl]{$z_2$}}
\put(51,9){\makebox(0,0)[tl]{$x_2$}}
\put(9,9){\makebox(0,0)[tr]{$w_1$}}
\put(51,31){\makebox(0,0)[bl]{$w_2$}}

\end{picture}
\hfil
\begin{picture}(60,33)
\put(30,3){\makebox(0,0)[t]{$P_3\setminus\{x_0\}$}}

\put(10,10){\circle*{3}}
\put(10,30){\circle*{3}}
\put(20,20){\circle*{3}}
\put(30,10){\circle*{3}}
\put(30,20){\circle*{3}}
\put(30,30){\circle*{3}}
\put(40,20){\circle*{3}}
\put(50,10){\circle*{3}}
\put(50,30){\circle*{3}}

\put(10,10){\line(0,1){20}}
\put(10,30){\line(1,-1){20}}
\put(30,10){\line(0,1){20}}
\put(30,30){\line(1,-1){20}}
\put(50,10){\line(0,1){20}}

\put(9,31){\makebox(0,0)[br]{$y_0$}}
\put(21,21){\makebox(0,0)[bl]{$z_1$}}
\put(31,9){\makebox(0,0)[tl]{$x_1$}}
\put(29,31){\makebox(0,0)[br]{$y_1$}}
\put(31,21){\makebox(0,0)[bl]{$z_2$}}
\put(41,21){\makebox(0,0)[bl]{$z_3$}}
\put(51,9){\makebox(0,0)[tl]{$x_2$}}

\end{picture}
}
$$

As for $P_1$, there are two \qmonered\ \sscn:
$y$, $x$ and the empty sequence.
Since
$
P_1\setminus G_{(y,x)}^{(-1)}=P_1\setminus G_{()}^{(-1)}=\{z\}
$,
$\dim C^{(-1)}_{(y,x)}=2$ and 
$\dim C^{(-1)}_{()}=1$
and the anticanonical analytic spread is $3$ and it comes from 
the \qmonered\ \scn\ $y$, $x$.
By the definition of $C^{(-1)}_{(y,x)}$ and $C^{(-1)}_{()}$,
we see that $C^{(-1)}_{()}\subset C^{(-1)}_{(y,x)}$.

As for $P_2$, there are four \qmonered\ \sscn:
$(y_0, x_1, y_1, x_2)$;
$(y_0, x_1)$; $(y_1, x_2)$ 
and the empty sequence.
$P_2\setminus G_{(y_0,x_1,y_1,x_2)}^{(-1)}=\{z_1,z_2\}$,
$P_2\setminus G_{(y_0,x_1)}^{(-1)}=\{z_1,z_2, x_2,w_2\}$,
$P_2\setminus G_{(y_1,x_2)}^{(-1)}=\{w_1, y_0, z_1,z_2\}$ 
and
$P_2\setminus G_{()}^{(-1)}=\{z_1,z_2\}$.
Therefore, we see that $\dim C^{(-1)}_{(y_0,x_1,y_1,x_2)}=4$,
$\dim C^{(-1)}_{(y_0,x_1)}=
\dim C^{(-1)}_{(y_1,x_2)}=5$ and
$\dim C^{(-1)}_{()}=2$, the anticanonical analytic spread is 6 and it comes
from the \qmonered\ \sscn\ $y_0$, $x_1$  and $y_1$, $x_2$.
Further, we see that $C^{(-1)}_{(y_0,x_1)}\cap C^{(-1)}_{(y_1,x_2)}=C^{(-1)}_{()}$.
Moreover,
$C^{(-1)}_{(y_0,x_1)}\cap
C^{(-1)}_{(y_0,x_1,y_1,x_2)}
=\{\nu\colon P_2^+\to\RRR\mid\nu\in C^{(-1)}_{(y_0,x_1,y_1,x_2)}$, $\nu(y_1)=-1\}
=\{\nu\colon P_2^+\to\RRR\mid\nu\in C^{(-1)}_{(y_0,x_1)}$, $\nu(w_2)=-1$, $\nu_(x_2)=-2\}$.
This is a 3-dimensional face of both $C^{(-1)}_{(y_0,x_1,y_1,x_2)}$ and
$C^{(-1)}_{(y_0,x_1)}$.
A similar fact holds for 
$C^{(-1)}_{(y_1,x_2)}\cap
C^{(-1)}_{(y_0,x_1,y_1,x_2)}$.

As for $P_3$, there are two \qmonered\ \sscn:
$y_0$, $x_1$, $y_1$, $x_2$
and the empty sequence.
$P_3\setminus G_{(y_0,x_1,y_1,x_2)}^{(-1)}=\{z_1,z_3\}$
and
$P_3\setminus G_{()}^{(-1)}=\{z_1,z_2, z_3,x_1,y_1\}$.
Therefore, 
$\dim C^{(-1)}_{(y_0,x_1,y_1,x_2)}=4$, $\dim C^{(-1)}_{()}=5$,
the anticanonical analytic spread is 6 and it comes
from the empty sequence.
Further,
$C^{(-1)}_{(y_0,x_1,y_1,x_2)}\cap  C^{(-1)}_{()}=
\{\nu\colon P_3^+\to\RRR\mid\nu\in C^{(-1)}_{()}$, $\nu(y_1)=\nu(z_2)+1=\nu(x_1)+2\}
=\{\nu\colon P_3^+\to\RRR\mid\nu\in C^{(-1)}_{(y_0,x_1,y_1,x_2)}$, $\nu(y_0)=-1\}$,
which is a 3-dimensional face of both 
$C^{(-1)}_{(y_0,x_1,y_1,x_2)}$ and $C^{(-1)}_{()}$.
\end{example}

\begin{prob}\rm
Determine the polytopal complex structure of
$
\{C\mid$ there is $(y_0,\ldots, x_t)\in N^{(\epsilon)}$ such that
$C$ is a face of $C^{(\epsilon)}_{(y_0,\ldots,x_t)}\}$.
\end{prob}

Suppose that $T^{\nu_1}T^{\nu_2}\neq0$ in 
$\bigoplus_{n\geq0}\omega^{(n\epsilon)}/\mmmm\omega^{(n\epsilon)}$.
Set $T^{\nu_j}\in\omega^{(n_j\epsilon)}$ for $j=1$, $2$, 
$n=n_1+n_2$ and $\nu=\nu_1+\nu_2$.
Then by Lemma \ref{lem:mo3.3}, we see that there is a \qnered\ sequence 
$y_0,x_1,\ldots, y_{t-1},x_t$ with \condn\
such that $\nu(x_i)-\nu(y_i)=\qne\dist(x_i,y_i)
=n\qe\dist(x_i,y_i)$ for $0\leq i\leq t$, where $y_t\define\infty$.
Since $\nu_j(x_i)-\nu_j(y_i)\geq n_j\qe\dist(x_i,y_i)$ for any $i$ and $j$,
we see that 
$$
\nu_j(x_i)-\nu_j(y_i)=n_j\qe\dist(x_i,y_i)
$$
for any $i$ and $j$.
Thus, we see that $\nu_1$ and ${\nu_2}$ are elements of 
$C^{(\epsilon)}_{(y_0,x_1,\ldots,y_{t-1},x_t)}$ and the product 
$T^{\nu_1}T^{\nu_2}=T^\nu$ is the one in
$R^{(\epsilon)}_{(y_0,x_1,\ldots,y_{t-1},x_t)}$.

In other words, if we set $\Gamma\define
\{C\mid$ there is $(y_0,\ldots, x_t)\in N^{(\epsilon)}$ such that
$C$ is a face of $C^{(\epsilon)}_{(y_0,\ldots,x_t)}\}$,
then the product of the elements $T^\nu$ and $T^{\nu'}$ in the right hand side of the equation
\refeq{eq:fc ne} is zero if there is no 
facet of $\Gamma$ containing both $\nu$ and $\nu'$ and 
the one in the Ehrhart ring of $C$ if there is a facet $C$ containing both $\nu$ and $\nu'$.

A \sr\ ring is a ring of this kind over a simplicial complex whose facets have
normalized volume 1.
On account of this fact, we propose the following.

\begin{prob}\rm
Establish a theory of polytopal complex version of \sr\ rings.
\end{prob}
Ishida \cite{ish} studied \cm\ and \gor\ properties of this kind of rings defined
by subcomplexes of boundary complex of convex polytopes.


\section{Complexity of a graded ring and Frobenius complexity}

\mylabel{sec:comp}

From now on, we use the term ring to express a not necessarily 
commutative ring with identity.

First we define the complexity of a graded ring.

\begin{definition}\rm
\mylabel{def:comp}
Let $A=\bigoplus_{n\geq 0} A_n$ be an $\NNN$-graded ring.
For $e\geq 0$, we denote by $G_e(A)$ the subring of $A$ generated by homogeneous elements
with degree at most $e$ over $A_0$.
For $e\geq 1$,
we 
denote by $c_e(A)$ the minimal number of elements which generate
$A_e/G_{e-1}(A)_e$ as a two sided $A_0$-module.
If $c_e(A)$ is finite for any $e$, we say that $A$ is  \dwfg.
Suppose that $A$ is \dwfg.
We define the complexity $\cx(A)$ of $A$ by
$$
\cx(A)\define\inf\{n\in\RRR_{>0}\mid c_e(A)=O(n^e)\ (e\to\infty)\}
$$
if $\{n\in\RRR_{>0}\mid c_e(A)=O(n^e)\ (e\to\infty)\}\neq\emptyset$.
We define $\cx(A)\define\infty$ if 
$\{n\in\RRR_{>0}\mid c_e(A)=O(n^e)\ (e\to\infty)\}=\emptyset$.
\end{definition}
$O$ in the above definition is the Landau symbol,
i.e.,
$g(x)=O(f(x))$ $(x\to\infty)$ means that there is a positive real number $K$
such that $|g(x)|<K|f(x)|$ for $x\gg0$.
We denote $g(x)\neq O(f(x))$ $(x\to \infty)$ if
$g(x)= O(f(x))$ $(x\to \infty)$ does not hold.
We state over which variable the limit is taken when
using Landau symbol, except the case that there is completely no
fear of confusion,
in order to avoid confusion.

Enescu-Yao \cite[Definition 2.9]{ey1} defined the notion of left $R$-skew
algebra.
We refine their definition and define the notion of strong left $R$-skew algebra.

\begin{definition}\rm
Let $R$ be a commutative ring and $A=\bigoplus_{n\geq0}A_n$  a graded ring.
Suppose that a ring homomorphism $R\to A_0$ is fixed.
We say that $A$ is a strong left $R$-skew algebra if $aI\subset Ia$ for any
homogeneous element $a\in A$ and any ideal $I\subset R$.
\end{definition}

\begin{remark}\rm
Let $R$ be a commutative ring,
$I$ an ideal of $R$ and $A=\bigoplus_{n\geq0}A_n$ a strong left $R$-skew algebra.
Then $IA=\bigoplus_{n\geq 0}IA_n$ is a two sided ideal of $A$ and
$A/IA=\bigoplus_{n\geq0}A_n/IA_n$ has naturally a graded ring structure.
\end{remark}

\begin{remark}\rm
\mylabel{rem:mod cx}
Suppose that $A_0$ is commutative and $A=\bigoplus_{n\geq0}A_n$ is
a \dwfg\ strong left $A_0$-skew algebra.
Then $c_e(A)$ is equal to the minimal number of generators of $A_e/G_{e-1}(A)_e$
as a left $A_0$-module.
Moreover, if $A_0$ is a local ring with maximal ideal $\mmmm$, then
$A/\mmmm A=\bigoplus_{n\geq0}A_n/\mmmm A_n$ and $c_e(A)=c_e(A/\mmmm A)$.
In particular, $c_e(A)$ is equal to the dimension of
$A_e/(G_{e-1}(A)_e+\mmmm A_e)=(A/\mmmm A)_e/G_{e-1}(A/\mmmm A)_e$
as a vector space over $A_0/\mmmm$.
\end{remark}

Next we recall the definition of Frobenius complexity.
Let $R$ be a commutative ring with prime characteristic $p$ and $M$ 
an $R$-module.
We denote by $\prehat{e}M$ the $R$-module whose additive group structure
is that of $M$ and the action of $R$ is defined by $e$ times iterated
Frobenius map, i.e. $r\cdot m=r^{p^e}m$ for $r\in R$ and $m\in M$,
where the $R$-action of right hand side is the original $R$-module
structure of $M$.
Note that for $\varphi\in\hom_R(M,\prehat{e}M)$ and $\psi\in\hom(M,\prehat{e'}M)$,
$\psi\circ\varphi\in\hom_R(M,\prehat{e+e'}M)$.
On account of this fact, we state the following.

\begin{definition}\rm
\mylabel{def:frob op}
Let $R$ and $M$ as above.
We set $\FFFFF^e(M)\define\hom(M,\prehat{e}M)$
and
$$
\FFFFF(M)\define\bigoplus_{e\geq0}\FFFFF^e(M).
$$
We call $\FFFFF(M)$ the ring of Frobenius operators on $M$.
The multiplication on $\FFFFF(M)$ is defined by composition of maps.
\end{definition}

\begin{definition}[{\cite[Definition 2.13]{ey1}}]
\rm
\mylabel{def:cxf}
Let $(R,\mmmm,k)$ be a commutative Noetherian local ring and $E$ the injective hull of $k$.
Then the Frobenius complexity $\cxf(R)$ of $R$ is defined by
$$
\cxf(R)\define\log_p(\cx(\FFFFF(E)))
$$
if $\FFFFF(E)$ is not finitely generated over $\FFFFF^0(E)$.
If $\FFFFF(E)$ is finitely generated over $\FFFFF^0(E)$,
we define $\cxf(R)\define-\infty$.
For an $\NNN$-graded commutative Noetherian ring $R=\bigoplus_{n\geq0}R_n$ with $R_0$
a field, we define the Frobenius complexity of $R$ to be that of 
$\mmmm$-adic completion of $R$, where $\mmmm=\bigoplus_{n>0}R_n$.
\end{definition}

\begin{remark}[{cf.\ \cite[3.3. Proposition.]{ls}}]
\rm
Let $(R,\mmmm,k)$ be as above and $\hat R$ the $\mmmm$-adic completion of $R$.
Then $(\hat R, \mmmm\hat R,k)$ is a local ring,
$E_R(k)=E_{\hat R}(k)$ and
$\hom_R(E_R(k),\prehat{e}E_R(k))
=\hom_{\hat R}(E_{\hat R}(k),\prehat{e}E_{\hat R}(k))$.
Thus, we see that Frobenius complexity does not vary by taking completion.
\end{remark}


\section{T-construction and T-complexity}

\mylabel{sec:t const}

Katzman et al. introduced an important graded ring construction method
from a commutative graded ring with prime characteristic.
We first recall their definition.

\begin{definition}[{\cite[Definition 2.1]{kssz}}]
\rm
\mylabel{def:tc}
Let $R=\bigoplus_{n\geq0}R_n$ be an $\NNN$-graded commutative ring with 
characteristic $p$.
We set $T(R)_e\define R_{p^e-1}$ for $e\geq 0$ and
$$
T(R)\define\bigoplus_{e\geq0}T(R)_e.
$$
The multiplication in $T(R)$ is defined by
$a\ast b\define ab^{p^e}$ for $a\in T(R)_e$ and $b\in T(R)_{e'}$
(the right hand side is the original product in $R$).
\end{definition}
%
%
Note that since $R$ is a commutative ring with characteristic $p$,
the product $\ast$ satisfies distributive law.
Thus, $T(R)$ is an $\NNN$-graded ring.

Next we make the following.

\begin{definition}
\rm
\mylabel{def:tcx}
In the situation of Definition \ref{def:tc}, we set
$$
\tcx(R)\define\log_p\cx(T(R))
$$
if $T(R)$ is not finitely generated over $T(R)_0$ and 
$\tcx(R)\define-\infty$ if $T(R)$ is finitely generated over $T(R)_0$
and call $\tcx(R)$ the T-complexity of $R$.
\end{definition}
Next we recall the following.

\begin{definition}[{\cite[Definition 3.2]{kssz}}]
\rm
Let $R$ be a commutative Noetherian normal ring that is either complete local or 
$\NNN$-graded and finitely generated over a field $R_0$.
Let $\omega$ denote the canonical ideal of $R$ and for $m\in\ZZZ$, let $\omega^{(m)}$
be the $m$-th power of $\omega$ in $\Div(R)$.
Then 
$$
\RRRRR\define\bigoplus_{n\geq0}\omega^{(-n)}
$$
is called the anticanonical cover of $R$.
\end{definition}
Note that $\RRRRR$ is a commutative graded ring with $\RRRRR_0=R$.
In particular, $\RRRRR$ and $R$ have the same characteristic.

Now we recall a crucially important result of Katzman et al.

\begin{fact}[{\cite[Theorem 3.3]{kssz}}]
\mylabel{fact:kssz}
Let $(R,\mmmm)$ be a commutative \cm\ normal complete local ring of characteristic $p$,
$E$ the injective hull of $R/\mmmm$, $\RRRRR$ the anticanonical cover of $R$.
Then there is an isomorphism of graded rings
$$
\FFFFF(E)\cong T(\RRRRR).
$$
\end{fact}
Note in the setting of Fact \ref{fact:kssz}, $T(\RRRRR)_0=R$ and
$T(\RRRRR)$ is a strong left $R$-skew algebra.

\begin{remark}\rm
\mylabel{rem:completion}
If $R$ is a normal excellent ring (e.g. a finitely generated commutative ring
over a field) and $\mmmm$ a maximal ideal of $R$, then the $\mmmm$-adic completion
$\hat R$ of $R$ is a complete normal local ring with maximal ideal $\mmmm \hat R$.
See e.g. \cite[(33.I) Theorem 79]{mat}.
In particular, if $R=\bigoplus_{n\geq0}R_n$ is a commutative normal graded ring
which is finitely generated over a field $R_0$, then the $\mmmm$-adic completion of $R$
is a normal complete local ring, where $\mmmm=\bigoplus_{n>0}R_n$.

Further, if $(R,\mmmm)$ is an excellent \cm\ normal local ring with canonical module or
finitely generated $\NNN$-graded \cm\ normal 
ring over a field $R_0$ and $\mmmm=\bigoplus_{n>0}R_n$,
then
$$
\omega_R^{(-n)}/\mmmm\omega_R^{(-n)}
=
\omega_{\hat R}^{(-n)}/\mmmm\omega_{\hat R}^{(-n)}
$$
for any $n\geq 0$, where $\hat R$ is the $\mmmm$-adic completion of $R$.
\end{remark}

In view of Remark \ref{rem:completion},
Fact \ref{fact:kssz}, Remark \ref{rem:mod cx} and 
equation \refeq{eq:fc ne},
we consider the T-complexity of Ehrhart rings.
First we note the following fact (cf. \cite[Proposition 2.6]{pag}).

\begin{lemma}
\mylabel{lem:u bound}
Let $R=\bigoplus_{n\geq0}R_n$ be a commutative Noetherian $\NNN$-graded ring
with $R_0$ a field of characteristic $p$.
Then
$$
\tcx(R)\leq\dim R-1.
$$
\end{lemma}
\begin{proof}
Set $d=\dim R$.
Since $R$ is Noetherian, there is a polynomial $f(n)$ of $n$ with degree $d-1$
such that
$$
\dim_{R_0}R_n\leq f(n)\quad\mbox{for $n\gg0$.}
$$
Therefore, $c_e(T(R))\leq\dim_{R_0}R_{p^e-1}\leq f(p^e-1)=O(p^{(d-1)e})$
$(e\to \infty)$.
Thus, $\tcx(R)\leq d-1$.
\end{proof}

Now we state the following.

\begin{lemma}
\mylabel{lem:msj5}
Let $d$ be an integer with $d\geq 2$ and
let $\Delta$ be an integral convex polytope in $\RRR^d$ such that
$\dim\Delta=d$ and 
$\Delta\subset\{(x_1,\ldots, x_d)\in \RRR^d\mid
x_i\geq0$ $(1\leq i\leq d)$, $\sum_{i=1}^d x_i\leq d-1\}$, and
$R$ the Ehrhart ring defined by $\Delta$ with base field $\KKK$
of characteristic $p$.
Then $\lim_{p\to\infty}\tcx(R)=d$
\end{lemma}
\begin{proof}
By Lemma \ref{lem:u bound}, we see that $\tcx(R)\leq d$ for any $p$.

Now we prove that $\liminf_{p\to\infty}\tcx(R)\geq d$.
Let $e'$ be a positive integer and $x=(x_1, \ldots, x_d)\in\NNN^d$.
We set $y'_i$ the remainder when $x_i$ is divided by $p^{e'}$ for $1\leq i\leq d$
and $y'=(y'_1,\ldots, y'_d)$.
If $y'_i\geq p^{e'}-\lfloor p^{e'}/d\rfloor$ for any $i$, then
$y'\not\in (p^{e'}-1)\Delta$ since 
$\sum_{i=1}^d y'_i\geq (d-1)p^{e'}$
and $(p^{e'}-1)\Delta\subset\{(w_1,\ldots,w_d)\in\RRR^d\mid\sum_{i=1}^d w_i\leq(d-1)(p^{e'}-1)\}$.
Thus, there are no $y=(y_1,\ldots, y_d)\in(p^{e'}-1)\Delta\cap\ZZZ^d$ and $z\in\ZZZ^d$ such that
$$
x=y+p^{e'}z,
$$
since $y_i\equiv y'_i\pmod{p^{e'}}$ for $1\leq i\leq d$.
Therefore for $e\geq 2$ and $x=(x_1, \ldots, x_d)\in (p^e-1)\Delta\cap\ZZZ^d$,
if each digit of the position $1$, $p$, $p^2$, \ldots, $p^{e-2}$ in
base $p$ expansion
of $x_i$ is greater than or equals to $p-\lfloor p/d\rfloor$, then
there are no $e'$, $x'$ and $x''$ such that 
$0<e'<e$,
$x=x'+p^{e'}x''$,
$x'\in(p^{e'}-1)\Delta\cap\ZZZ^d$ and
$x''\in(p^{e-e'}-1)\Delta\cap\ZZZ^d$.
In fact, 
since the digit of $x_i$ of the position $p^{e'-1}$ is greater than or equals to
$p-\lfloor p/d\rfloor$, we see that 
the remainder when $x_i$ is divided by $p^{e'}$ is
greater than or equals to
$p^{e'-1}(p-\lfloor p/d\rfloor)
\geq p^{e'}-\lfloor p^{e'}/d\rfloor$ for $1\leq i\leq d$.
This contradicts the fact proved above.

Since $\Delta$ has an interior point, there is $N>d$ 
such that if $p>N$, then there are
positive integers
$a_1$, \ldots, $a_d$ such that
$[a_1/p,(a_1+2)/p]\times\cdots\times[a_d/p,(a_d+2)/p]\subset\Delta$.

Let $e$ be an integer with $e\geq 2$.
Since $a_ip^{e-1}>(p^e-1)a_i/p$ and
$(a_i+1)p^{e-1}-1<(p^e-1)(a_i+2)/p$,
we see that
$[a_1p^{e-1},(a_1+1)p^{e-1}-1]\times\cdots\times[a_dp^{e-1},(a_d+1)p^{e-1}-1]\subset(p^e-1)\Delta$.
If $x=(x_1,\ldots,x_d)\in
[a_1p^{e-1},(a_1+1)p^{e-1}-1]\times\cdots\times[a_dp^{e-1},(a_d+1)p^{e-1}-1]$
and each digit of the position $1$, $p$, $p^2$, \ldots, $p^{e-2}$ of base $p$ expansion
of
$x_i$ is greater than or equals to $p-\lfloor p/d\rfloor$ for any $i$,
then $x\in(p^e-1)\Delta$ and there are no $e'$, $x'$ and $x''$ such that
$0<e'<e$,
$x=x'+p^{e'}x''$,
$x'\in(p^{e'}-1)\Delta\cap\ZZZ^d$ and
$x''\in(p^{e-e'}-1)\Delta\cap\ZZZ^d$.

Since there are $(\lfloor p/d\rfloor)^{d(e-1)}$ choices of $x$,
we see that $c_e(T(R))\geq(\lfloor p/d\rfloor)^{d(e-1)}$ if $p>N$.
Thus $\cx(T(R))\geq(\lfloor p/d\rfloor)^d$ and
$
\liminf_{p\to\infty}\tcx(R)=\liminf_{p\to\infty}\log_p \cx(T(R))\geq d
$.
\end{proof}

Next we state a lemma which is crucial to apply Lemma \ref{lem:msj5} to
more general polytope.
We first state the definition of symbols.

\begin{definition}\rm
Let $\Delta$ be a convex polytope in $\RRR^d$ with $\dim\Delta=d$ and 
$\delta$ a positive real number.
We denote by $\partial\Delta$ the boundary of $\Delta$.
We set
$\partial'_{\delta}(\Delta)\define\{P\in\Delta\mid$
the distance between $P$ and $\partial\Delta$ is less than $\delta\}$,
$\partial''_{\delta}(\Delta)\define\{P\in\Delta\mid$
the distance between $P$ and $\partial\Delta$ is equal to $\delta\}$
and
$\interior'_{\delta}(\Delta)\define\{P\in\Delta\mid$
the distance between $P$ and $\partial\Delta$ is greater than $\delta\}$.
\end{definition}
Now we state the following.

\begin{lemma}
\mylabel{lem:msj6}
Let $\Delta'$ be a $d+1$ dimensional integral convex polytope in $\RRR^{d+1}$.
Set $\Delta=\{(x_1,\ldots,x_d)\in\RRR^d\mid\exists y\in\RRR; (x_1,\ldots, x_d, y)\in\Delta'\}$
and let $R$ (resp.\ $R'$) be the Ehrhart ring of $\Delta$ (resp.\ $\Delta'$)
in a Laurent polynomial ring $\KKK[X_1^{\pm1},\ldots, X_d^{\pm1},T]$
(resp.\ $\KKK[X_1^{\pm1},\ldots, X_{d+1}^{\pm1},T]$),
where $\KKK$ is a field of characteristic $p$.
Then 
$\lim_{p\to\infty}\tcx(R')=d+1$
if $\lim_{p\to\infty}\tcx(R)=d$.
\end{lemma}
\begin{proof}
First we note that by Lemma \ref{lem:u bound}, $\tcx(R')\leq d+1$ for any $p$.

Next we prove that $\liminf_{p\to\infty}\tcx(R')\geq d+1$.
Let $\epsilon$ be an arbitrary real number with $0<\epsilon<1$.
Since $\lim_{p\to\infty}\tcx(R)=d$, we see that there exists $N$ such that
if $p>N$, then $\tcx(R)>d-\epsilon$.
Let $p$ be such a prime number.
For a positive integer $e$, we set
$H_e(R)=\{P\in(p^e-1)\Delta\cap\ZZZ^d\mid
X^PT^{p^e-1}\not\in G_{e-1}(T(R))\}$,
where $X^P\define X_1^{p_1}\cdots X_d^{p_d}$ for $P=(p_1,\ldots, p_d)$.
We define $H_e(R')$ similarly.
Then
$H_e(R)=\{P\in(p^e-1)\Delta\cap\ZZZ^d\mid$
there are no $e'$, $P_1$ and $P_2$ such that $0<e'<e$, $P_1\in(p^{e'}-1)\Delta\cap\ZZZ^d$
and $P_2\in(p^{e-e'}-1)\Delta\cap\ZZZ^d$ with $P=P_1+p^{e'}P_2\}$.
In particular,
$(H_e(R)\times\ZZZ)\cap(p^{e}-1)\Delta'\subset H_e(R')$.

Set $r=1-\epsilon$.
Then $\#(\partial'_{p^{re}}((p^e-1)\Delta)\cap\ZZZ^d)=O(p^{(r+d-1)e})=O(p^{(d-\epsilon)e})$
$(e\to\infty)$, since $\dim\partial\Delta=d-1$.
On the other hand, since $\tcx(R)>d-\epsilon$, we see that
$\#H_e(R)\neq O(p^{(d-\epsilon)e})$ $(e\to\infty)$.
Therefore,
\begin{equation}
\mylabel{eq:setminus}
\#(H_e(R)\setminus(\partial'_{p^{re}}(p^e-1)\Delta))\neq O(p^{(d-\epsilon)e})\quad
(e\to\infty).
\end{equation}

For $x=(x_1,\ldots, x_d)\in\Delta$, set
$h(x)\define\max\{y-z\mid(x_1,\ldots,x_d,y), (x_1,\ldots,x_d,z)\in\Delta'\}$.
Then, since $\Delta'$ is the intersection of finite number of halfspaces, we see that
there exist  positive real numbers $a$ and $\delta$ such that
if $0<\delta'\leq\delta$ and $x\in\partial''_{\delta'}(\Delta)$, then
$h(x)\geq a\delta'$.
For these $a$ and $\delta$, we see that 
$h(P)\geq a\delta'$ for any $0<\delta'\leq \delta$ and $P\in\interior'_{\delta'}(\Delta)$,
since $\Delta'$ is convex.
Thus, we see by \refeq{eq:setminus} that
\begin{eqnarray*}
c_e(T(R'))&=&\#H_e(R')\\
&\geq&\#((H_e(R)\setminus\partial'_{p^{re}}(p^e-1)\Delta)\times\ZZZ\cap(p^e-1)\Delta')\\
&\geq&\#((H_e(R)\setminus\partial'_{p^{re}}(p^e-1)\Delta))(ap^{re}-1)\\
&\neq&O(p^{(d-\epsilon+r)e})\ (e\to\infty),
\end{eqnarray*}
since $p^{re}/(p^e-1)\leq\delta$ for $e\gg0$.
Therefore, $\tcx(R')\geq d-\epsilon+r=d+1-2\epsilon$.

Since $\epsilon$ is an arbitrary real number with $0<\epsilon<1$, we see that
$\liminf_{p\to\infty}\tcx(R')\geq d+1$.
\end{proof}


\section{T-complexities of fiber cones and limit Frobenius complexity
of Hibi rings}

\mylabel{sec:frob}

In this section, we consider the limit of Frobenius complexities of 
Hibi rings as $p\to \infty$, where $p$ is the characteristic of the base field.
Recall that $H$ is a finite distributive lattice with minimal element $x_0$,
$P$ the set of \joinirred\ elements of $H$, $\rkh$ the Hibi ring over
a field $\KKK$ on $H$  and $\omega$ the canonical module of $\rkh$.
In this section, we assume that $\KKK$ is a field of characteristic $p$.

In view of Fact \ref{fact:kssz}, Remarks \ref{rem:mod cx} and \ref{rem:completion},
we consider the T-complexity of fiber cones of the anticanonical ideal of $\rkh$.
We use the notation of \S\ref{sec:anal spread}.
First we state the following.

\begin{lemma}
\mylabel{lem:ce max}
$$
c_e(T(\bigoplus_{n\geq0}\omega^{(-n)}/\mmmm\omega^{(-n)}))
\geq
c_e(T(R^{(-1)}_{(y_0,x_1,\ldots,y_{t-1},x_t)}))
$$
for any $e>0$ and any $(y_0, x_1, \ldots, y_{t-1}, x_t)\in N^{(-1)}$.
\end{lemma}
\begin{proof}
Set $R=R^{(-1)}_{(y_0,x_1,\ldots,y_{t-1},x_t)}$.
Suppose that $T^\nu\in T(R)_e$ and $T^\nu\not\in G_{e-1}(T(R))_e$.
Since $T^\nu\in T(\bigoplus_{n\geq0}\omega^{(-n)}/\mmmm\omega^{(-n)})$
by equation \refeq{eq:fc ne},
it is enough to show that
$T^\nu\not\in G_{e-1}(T(\bigoplus_{n\geq0}\omega^{(-n)}/\mmmm\omega^{(-n)}))_e$.
Assume the contrary.
Then there exist $e'\in\NNN$, $\nu'$ and $\nu''$ such that
$0<e'<e$, 
$T^{\nu'}\in\omega^{(1-p^{e'})}/\mmmm\omega^{(1-p^{e'})}$,
$T^{\nu''}\in\omega^{(1-p^{e-e'})}/\mmmm\omega^{(1-p^{e-e'})}$
and
$T^\nu=T^{\nu'}\ast T^{\nu''}=T^{\nu'}(T^{\nu''})^{p^{e'}}$.
Since
$\nu'(y_i)-\nu'(x_i)\geq \qdist{1-p^{e'}}(x_i,y_i)$,
$\nu''(y_i)-\nu''(x_i)\geq \qdist{1-p^{e-e'}}(x_i,y_i)$,
$\nu(y_i)-\nu(x_i)= \qdist{1-p^{e}}(x_i,y_i)$
and $\nu=\nu'+p^{e'}\nu''$,
we see that
$\nu'(y_i)-\nu'(x_i)= \qdist{1-p^{e'}}(x_i,y_i)$
and
$\nu''(y_i)-\nu''(x_i)= \qdist{1-p^{e-e'}}(x_i,y_i)$
for any $i$.
Therefore,
$T^{\nu'}\in T(R)_{e'}$,
$T^{\nu''}\in T(R)_{e-e'}$
and
$T^{\nu}=T^{\nu'}\ast T^{\nu''}$.
This contradicts to the assumption that $T^\nu\not\in G_{e-1}(T(R))_e$.
\end{proof}

Before going further, we note that, by Remark \ref{rem:ce full}, 
$C^{(-1)}_{(y_0,x_1,\ldots,y_{t-1},x_t)}$ is essentially a
full dimensional convex polytope in
$\RRR^{(P\setminus 
G^{(-1)}_{(y_0,x_1,\ldots, y_{t-1},x_t)})\cup\{y_0,\ldots, y_{t-1}\}}$.

Next we state the following.

\begin{lemma}
\mylabel{lem:max dim tcx}
Assume that $P$ is not pure
and let $y_0$, $x_1$, \ldots, $y_{t-1}$, $x_t$ a
\qmonered\  \scn\ such that
$\dim C^{(-1)}_{(y_0,x_1,\ldots, y_{t-1},x_t)}=
\max_{(y'_0,x'_1,\ldots,y'_{t'-1},x'_{t'})\in
 N^{(-1)}} \dim C^{(-1)}_{(y'_0,x'_1,\ldots, y'_{t'-1},x'_{t'})}$.
Then
$\lim_{p\to\infty}\tcx (R^{(-1)}_{(y_0,x_1,\ldots,y_{t-1},x_t)})=
\dim C^{(-1)}_{(y_0,x_1,\ldots, y_{t-1},x_t)}$.
\end{lemma}
\begin{proof}
By Lemmas \ref{lem:msj5} and \ref{lem:msj6}, it is enough to show that there
is a projection 
of Euclidean space
to a coordinate subspace 
whose image $\Delta$ 
of 
$C^{(-1)}_{(y_0,x_1,\ldots, y_{t-1},x_t)}$
satisfies the condition of Lemma \ref{lem:msj5}
under unimodular transformation.
We set $y_t=\infty$.

First consider the case where $t=0$.
Since $P$ is not pure, 
$\dim C^{(-1)}_{()}\geq 1$ by 
Corollary \ref{cor:non min max}.
Thus, 
there is a sequence of elements
$z_0$, $z_1$, \ldots, $z_u$ in $P^+$ such that
$u\geq 2$,
$z_0\covered z_1\covered \cdots\covered z_u$,
$z_0$, $z_u\in G^{(-1)}_{()}$ and
$z_i\not\in G_{()}^{(-1)}$ for $1\leq i\leq u-1$.
Here, we claim that
$$
\dist(z_0,\infty)-\dist(z_u,\infty)\geq2.
$$
Assume the contrary.
Then, since $\dist(z_0,z_u)\geq 2$ and
$\dist(x_0,z_0)+\dist(z_0,\infty)=\dist(x_0,\infty)$,
we see that
\begin{eqnarray*}
&&\dist(x_0,z_0)+\dist(z_0,z_u)+\dist(z_u,\infty)\\
&=&\dist(x_0,\infty)-\dist(z_0,\infty)+\dist(z_0,z_u)+\dist(z_u,\infty)\\
&=&\dist(x_0,\infty)-(\dist(z_0,\infty)-\dist(z_u,\infty))+\dist(z_0,z_u)\\
&>&\dist(x_0,\infty).
\end{eqnarray*}
Since $\dist(x_0,z_0)+\dist(z_0,\infty)=\dist(x_0,z_u)+\dist(z_u,\infty)=\dist(x_0,\infty)$,
we see
by the above inequality
 that $z_0$, $z_u\not\in\{x_0,\infty\}$, i.e.
$z_u$, $z_0$ is a \scn.
Further, since
\begin{eqnarray*}
&&\qmone(x_0,z_u,z_0,\infty)\\
&=&-\dist(x_0,z_u)+\dist(z_0,z_u)-\dist(z_0,\infty)\\
&=&\dist(z_u,\infty)-\dist(x_0,\infty)+\dist(z_0,z_u)+\dist(x_0,z_0)-\dist(x_0,\infty)\\
&>&-\dist(x_0,\infty)\\
&=&\qmone\dist(x_0,\infty),
\end{eqnarray*}
$z_u$, $z_0$ is a \qmonered\ \scn.

Since $z_0$, $z_u\in G_{()}^{(-1)}$, we see that
$G^{(-1)}_{(z_u,z_0)}\subset G^{(-1)}_{()}$.
Thus, we see that
$$
\dim C^{(-1)}_{(z_u,z_0)}=
\#(P\setminus G^{(-1)}_{(z_u,z_0)})+1>
\#(P\setminus G^{(-1)}_{()})=\dim C_{()}^{(-1)}.
$$
This contradicts to the assumption.
Therefore,
$$
\dist(z_0,\infty)-\dist(z_u,\infty)\geq 2.
$$

Consider 
the image of $C^{(-1)}_{()}$ of composition of
the projection (restriction) $\RRR^P\to\RRR^{\{z_1,\ldots, z_{u-1}\}}$
and the transformation
$\xi(z_i)=\nu(z_{i-1})-\nu(z_i)+1$ for $1\leq i\leq u-1$.
This transformation $\nu\mapsto\xi$ is unimodular since it is a composition of a 
parallel translation
and a linear transformation whose representation matrix is an upper triangular matrix
with diagonal entries $-1$.
Note that $\nu(z_0)$ is independent of $\nu$
and therefore $\xi(z_1)$, \ldots, $\xi(z_{u-1})$ are defined by $\nu(z_1)$, \ldots, $\nu(z_{u-1})$.
Further, $\xi(z_i)\geq 0$ for $1\leq i\leq u-1$,
since
 by the definition of $C^{(-1)}_{()}$, $\nu(z_{i-1})-\nu(z_i)\geq-1$.
Moreover,
\begin{eqnarray*}
&&\sum_{i=1}^{u-1}\xi(z_i)\\
&=&\nu(z_0)-\nu(z_{u-1})+u-1\\
&\leq&\nu(z_0)-\nu(z_{u-1})+u-1+\nu(z_{u-1})-\nu(z_u)+1\\
&=&\qmone\dist(z_0,\infty)-\qmone\dist(z_u,\infty)+u\\
&\leq&u-2,
\end{eqnarray*}
since 
$\nu(z_0)=\qmone\dist(z_0,\infty)$,
$\nu(z_u)=\qmone\dist(z_u,\infty)$ and $\dist(z_0,\infty)-\dist(z_u,\infty)\geq 2$.
Therefore, the image 
of $C^{(-1)}_{()}$ 
by the composition of the projection and the above unimodular transformation 
satisfies the assumption of Lemma \ref{lem:msj5}.

Next consider the case where $t>0$.
First note that
$$
\dist(x_t,y_{t-1})+\dist(y_{t-1},\infty)>\dist(x_t,\infty).
$$
In fact, since $y_0$, $x_1$, \ldots, $y_{t-1}$, $x_t$ is a
\qmonered\ \scn, we see that
$$
\qmone(x_{t-1},y_{t-1},x_t,\infty)>\qmone\dist(x_{t-1},\infty),
$$
i.e.,
$$
\dist(x_{t-1},y_{t-1})-\dist(x_t,y_{t-1})+\dist(x_t,\infty)<\dist(x_{t-1},\infty).
$$
Since $\dist(x_{t-1},\infty)\leq\dist(x_{t-1},y_{t-1})+\dist(y_{t-1},\infty)$,
we see the inequality above.
In particular, $y_{t-1}\not\in G^{(-1)}_{t, (y_0,x_1,\ldots,y_{t-1},x_t)}$.

Now set $\ell=\dist(x_t,y_{t-1})$ and take elements $z_0$, $z_1$, \ldots, $z_\ell$
such that
$$
x_t=z_0\covered z_1\covered\cdots\covered z_\ell=y_{t-1}.
$$
We claim that $z_i\not\in G_{t,(y_0,x_1,\ldots, y_{t-1},x_t)}^{(-1)}$ for $1\leq i\leq \ell$.
Assume the contrary and take $i$ with $z_i\in G_{t,(y_0,x_1,\ldots,y_{t-1},x_t)}^{(-1)}$.
Then $i<\ell$ since $y_{t-1}\not\in G_{t,(y_0,x_1,\ldots, y_{t-1},x_t)}^{(-1)}$.
Therefore, it is easily verified that
$y_0$, $x_1$, \ldots, $x_{t-1}$, $y_{t-1}$, $z_i$ is a \qmonered\ \scn.
It is also verified that
$$
G_{(y_0,x_1,\ldots, y_{t-1},z_i)}^{(-1)}\subsetneq
G_{(y_0,x_1,\ldots, y_{t-1},x_t)}^{(-1)}.
$$
This contradicts to the maximality of
$\dim C_{(y_0,x_1,\ldots, y_{t-1},x_t)}^{(-1)}$.
Thus $z_i\not\in G_{t,(y_0,x_1,\ldots, y_{t-1},x_t)}^{(-1)}$ for $1\leq i\leq \ell$.
We see that $z_i\not\in G_{t-1,(y_0,x_1,\ldots, y_{t-1},x_t)}^{(-1)}$ for $0\leq i\leq \ell-1$
by the same way.
Therefore,
$$
z_i\not\in G_{(y_0,x_1,\ldots, y_{t-1},x_t)}^{(-1)}\quad\mbox{for $1\leq i\leq\ell-1$},
$$
since $y_0$, $x_1$, \ldots, $y_{t-1}$, $x_t$ satisfies \condn.

Now take elements $z_{\ell+1}$, \ldots, $z_{u-1}$, $z_u$ such that
$$
y_{t-1}=z_\ell\covered z_{\ell+1}\covered\cdots\covered z_{u-1}\covered z_u,
$$
$z_i\not\in G_{t,(y_0,x_1,\ldots, y_{t-1},x_t)}^{(-1)}$ for $\ell+1\leq i\leq u-1$
and
$z_u\in G_{t,(y_0,x_1,\ldots, y_{t-1},x_t)}^{(-1)}$
($u$ may equal to $\ell+1$).
Then $z_i\not\in G_{(y_0,x_1,\ldots, y_{t-1}, x_t)}^{(-1)}$ for $\ell+1\leq i\leq u-1$,
since $y_0$, $x_1$, \ldots, $y_{t-1}$, $x_t$ satisfies \condn.

Since $z_0=x_t$, $\dist(x_t,z_u)+\dist(z_u,\infty)=\dist(x_t,\infty)$ and $z_0$ is not 
covered by $z_u$, we see that
$$
\dist(z_0,\infty)-\dist(z_u,\infty)\geq 2.
$$
Therefore, we see by the same argument as in the case where $t=0$, that the image
of $C^{(-1)}_{(y_0,x_1,\ldots, y_{t-1},x_t)}$ 
of the composition of projection
$\RRR^P\to\RRR^{\{z_1, \ldots, z_{u-1}\}}$
and the same unimodular transformation as in the case where $t=0$ 
satisfies the condition of Lemma \ref{lem:msj5}.
\end{proof}

\begin{remark}\rm
Consider $P_1$ of Example \ref{ex:anal spread}.
$x$, $y\in G_{()}^{(-1)}$,
$\dist(x,\infty)-\dist(y,\infty)=1$ and $\dist(x,y)=2$.
Therefore,
$\dist(x,\infty)-\dist(y,\infty)
=\dist(x,y)$
does not hold in general for $x$, $y\in G^{(-1)}_{()}$ with $x<y$.
Thus, we need to prove 
$\dist(z_0,\infty)-\dist(z_u,\infty)\geq 2$
in the case of $t=0$ of the proof of Lemma \ref{lem:max dim tcx}.
In fact, $\dim C_{()}^{(-1)}=
\max_{(y'_0,x'_1,\ldots,y'_{t'-1},x'_{t'})\in
 N^{(-1)}} \dim C^{(-1)}_{(y'_0,x'_1,\ldots, y'_{t'-1},x'_{t'})}$
is essential.
\end{remark}

\begin{example}
\rm
Consider $P_1$ of Example \ref{ex:anal spread}.
There are following 3 minimal elements
$$
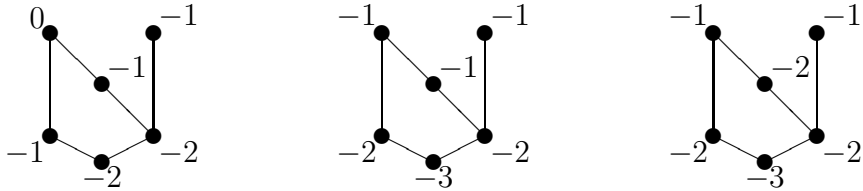
\vcenter{%
\unitlength=.005\textwidth\relax
\begin{picture}(40,33)

\put(20,5){\circle*{3}}
\put(10,10){\line(2,-1){10}}
\put(20,5){\line(2,1){10}}

\put(10,10){\circle*{3}}

\put(10,30){\circle*{3}}
\put(20,20){\circle*{3}}
\put(30,10){\circle*{3}}
\put(30,30){\circle*{3}}

\put(10,10){\line(0,1){20}}
\put(10,30){\line(1,-1){20}}
\put(30,10){\line(0,1){20}}

\put(9,31){\makebox(0,0)[br]{$0$}}
\put(9,9){\makebox(0,0)[tr]{$-1$}}
\put(21,21){\makebox(0,0)[bl]{$-1$}}
\put(31,9){\makebox(0,0)[tl]{$-2$}}
\put(31,31){\makebox(0,0)[bl]{$-1$}}
\put(20,4){\makebox(0,0)[t]{$-2$}}

\end{picture}
\hfil
\begin{picture}(40,33)

\put(20,5){\circle*{3}}
\put(10,10){\line(2,-1){10}}
\put(20,5){\line(2,1){10}}

\put(10,10){\circle*{3}}

\put(10,30){\circle*{3}}
\put(20,20){\circle*{3}}
\put(30,10){\circle*{3}}
\put(30,30){\circle*{3}}

\put(10,10){\line(0,1){20}}
\put(10,30){\line(1,-1){20}}
\put(30,10){\line(0,1){20}}

\put(31,31){\makebox(0,0)[bl]{$-1$}}
\put(31,9){\makebox(0,0)[tl]{$-2$}}
\put(21,21){\makebox(0,0)[bl]{$-1$}}
\put(9,31){\makebox(0,0)[br]{$-1$}}
\put(9,9){\makebox(0,0)[tr]{$-2$}}
\put(20,4){\makebox(0,0)[t]{$-3$}}

\end{picture}
\hfil
\begin{picture}(40,33)

\put(20,5){\circle*{3}}
\put(10,10){\line(2,-1){10}}
\put(20,5){\line(2,1){10}}

\put(10,10){\circle*{3}}

\put(10,30){\circle*{3}}
\put(20,20){\circle*{3}}
\put(30,10){\circle*{3}}
\put(30,30){\circle*{3}}

\put(10,10){\line(0,1){20}}
\put(10,30){\line(1,-1){20}}
\put(30,10){\line(0,1){20}}

\put(31,31){\makebox(0,0)[bl]{$-1$}}
\put(31,9){\makebox(0,0)[tl]{$-2$}}
\put(21,21){\makebox(0,0)[bl]{$-2$}}
\put(9,31){\makebox(0,0)[br]{$-1$}}
\put(9,9){\makebox(0,0)[tr]{$-2$}}
\put(20,4){\makebox(0,0)[t]{$-3$}}

\end{picture}
}
$$
of $\TTTTT^{(-1)}(P)$.
These are the vertices of $C^{(-1)}_{(y,x)}$.
Thus, projection of $C^{(-1)}_{(y,x)}$
to $\RRR^{\{y,z\}}$
is a rectangular equilateral triangle with normalized volume $1$.
\end{example}

Now we state the following.

\begin{thm}
\mylabel{thm:f comp}
If $\rkh$ is not \gor, then
$$
\lim_{p\to\infty}\cxf(\rkh)=\dim(\bigoplus_{n\geq0}\omega^{(-n)}/\mmmm\omega^{(-n)})-1.
$$
\end{thm}
\begin{proof}
By Fact \ref{fact:kssz}, Remarks \ref{rem:mod cx} and \ref{rem:completion},
we see that
$$
\cxf(\rkh)=\tcx(T(\bigoplus_{n\geq0}\omega^{(-n)}/\mmmm\omega^{(-n)})).
$$
On the other hand, 
By Lemma \ref{lem:u bound}, we see that
$$
\tcx(T(\bigoplus_{n\geq0}\omega^{(-n)}/\mmmm\omega^{(-n)}))
\leq\dim(\bigoplus_{n\geq0}\omega^{(-n)}/\mmmm\omega^{(-n)})-1
$$
for any $p$.
Further,
by Theorem \ref{thm:can anal spr}, Lemmas \ref{lem:ce max}
and \ref{lem:max dim tcx}, we see that
$$
\lim_{p\to\infty}
\tcx(T(\bigoplus_{n\geq0}\omega^{(-n)}/\mmmm\omega^{(-n)}))
\geq\dim(\bigoplus_{n\geq0}\omega^{(-n)}/\mmmm\omega^{(-n)})-1.
$$
\end{proof}

\end{document}